\documentclass[11pt,oneside,a4paper,leqno]{amsart}
\usepackage{amsmath}
\usepackage[hmargin=3cm, vmargin=3cm]{geometry}
\usepackage{amsfonts}
\usepackage{amsthm}

\usepackage{color} 
\usepackage{enumerate} 
\theoremstyle{plain}  
\newtheorem{theorem}{Theorem}[section] 
\newtheorem{proposition}[theorem]{Proposition}
\newtheorem{lemma}[theorem]{Lemma}
\newtheorem{corollary}[theorem]{Corollary}
\theoremstyle{definition}

\theoremstyle{remark}

\newcommand{\R}{\mathbb{R}}  

\newcommand{\la}{\lambda}
\newcommand{\equ}[1]{(\ref{#1})}
\newcommand{\ve}{\varepsilon}
\newcommand{\be}{\begin{equation}}
\newcommand{\ee}{\end{equation}}
\newcommand{\bea}{\begin{eqnarray}}
\newcommand{\eea}{\end{eqnarray}}

\begin{document}

\title[Sign-changing solutions for supercritical Bahri-Coron's problem]{Sign-changing blowing-up solutions for supercritical Bahri-Coron's problem}

\author{Monica Musso}
\address{\noindent M. Musso - Departamento de Matem\'atica,
Pontificia Universidad Catolica de Chile, Avda. Vicu\~na Mackenna
4860, Macul, Chile}\email{mmusso@mat.puc.cl}

\author{Juncheng Wei}
\address{\noindent J. Wei -Department of Mathematics, University of British Columbia, Vancouver, BC V6T1Z2, Canada }\email{jcwei@math.ubc.ca}

\thanks{ The research of the first  author
has been partly supported by Fondecyt Grant 1120151. The research of the second author is partially supported by NSERC of Canada.}

\maketitle

\vskip 0.2cm \arraycolsep1.5pt
\newtheorem{Lemma}{Lemma}[section]
\newtheorem{Theorem}{Theorem}[section]
\newtheorem{Definition}{Definition}[section]
\newtheorem{Proposition}{Proposition}[section]
\newtheorem{Remark}{Remark}[section]
\newtheorem{Corollary}{Corollary}[section]

\maketitle

\noindent {\bf Abstract}:
Let $\Omega$ be a bounded domain in $\R^n$, $n\ge 3$ with smooth
boundary $\partial\Omega$ and a small hole. We give the first example  of sign-changing {\it bubbling} solutions to the
nonlinear elliptic problem $$
                        -\Delta u=|u|^{{n+2\over n-2} +\ve -1  } u \, \, \mbox{ in } \Omega , \quad \quad
                        u=0 \mbox{ on } \partial \Omega,
$$ where $\ve$ is a small positive
parameter.  The basic cell in the construction is the sign-changing nodal solution to the critical Yamabe problem
$$ -\Delta w = |w|^{\frac{4}{n-2}} w, \  \ w \in {\mathcal D}^{1,2} (\R^n) $$
which has large number ($3n$) of kernels.

\vspace{3mm}

\maketitle

\section{Introduction}\label{intro}

Let $\Omega$ be a bounded domain in $\R^n$, $n\ge 3$ with smooth
boundary $\partial\Omega$. In this paper we establish  existence of a new type of {\it bubbling} solutions to the
nonlinear elliptic problem \bea
                        -\Delta u=|u|^{{n+2\over n-2} +\ve -1  } u \, \, \mbox{ in } \Omega , \quad \quad
                        u=0 \mbox{ on } \partial \Omega,
\label{P}\eea where $\ve$ is a small positive
parameter.

It is known that solvability for Problem
\begin{equation}
-\Delta u=|u|^{q-1} u \quad \mbox{ in }\Omega, \quad  \quad u=0 , \quad\mbox{ on } \partial \Omega ,\label{Q}
\end{equation}
is an elementary fact when $1< q< {n+2 \over n-2}$.
This is no longer the case for $q\ge \frac{n+2}{n-2}$
due to the loss of compactness of Sobolev embeddings. Our aim is to analyze
solutions exhibiting {\em bubbling behavior} to the above problem
when one lets the exponent $q$ approach $\frac{n+2}{n-2}$ from above.

Pohozaev \cite{Po} showed that
if $\Omega$ is strictly star-shaped then no solution of \equ{Q} exists
if  $q\ge \frac{n+2}{n-2}$.
In  contrast
 Kazdan and Warner \cite{KW} showed that,
if $\Omega$ is a radially symmetric annulus, $\Omega=\{ a<|x|<b\}$,
 there exists a
  radial positive  solution to Problem \equ{Q} for any exponent $q>1$.
Without symmetry the question is harder. This issue was
first considered by Coron
\cite{Co} who found
that \equ{Q} has a positive solution  when $q={n+2\over n-2}$  in any domain
exhibiting a small hole. Also a  second solution exists in Coron's setting, as shown in \cite{cw}, see also the results in \cite{cw1,cmp} and reference therein. The most general result concerning existence of positive solutions to \equ{Q} for $q={n+2\over n-2}$ is obtained by Bahri and Coron \cite{BC}: if
 some homology
group of $\Omega$ with coefficients in ${\bf Z}_2$ is not trivial,
then \equ{Q} has at least one positive solution, in particular in any
three-dimensional domain which is not contractible to a point.
Examples showing that this condition is actually not necessary for
solvability were found by Dancer \cite{Da}, Ding \cite{Di} and Passaseo \cite{Pa1}, for $q={n+2 \over n-2}$ and also for very super critical powers $q\geq {n+1 \over n-3} > {n+2 \over n-2}$, see \cite{Pa2}.
The question of existence
for super-critical powers close to critical has been addressed in \cite{dfm,dfm1,dfm2,pr,bengr}, where
existence of {\it positive} solutions
to \equ{Q} is established. These solutions become {\em unbounded}
as the exponent $q \downarrow \frac{n+2}{n-2}$ and they develop a {\it blowing-up} profile.

By a {\em blowing-up solution} for $\equ{Q}$ near the critical
exponent we mean an unbounded  sequence of solutions $u_j$ of \equ{Q} for
 $q=q_j \to {n+2\over n-2}$.
Setting
$$
M_j = \max_\Omega |u_j | =  |u_j (x_j ) |\to + \infty
$$
we see then that the
scaled function $$ v_j (y) =
M_j\, u_j ( x_j + M_j^{{(q_j -1)/ 2}}\, y), $$
 satisfies
$$
\Delta v_j  + |v_j|^{q_j -1 } v_j  = 0
$$
in the expanding domain $\Omega_j = M_j^{{(q_j -1)/ 2}}(\Omega - x_j)$.
Assuming for instance that $x_j$ stays away from the boundary of $\Omega$,
elliptic regularity implies that locally over compacts around the origin,
$v_j$ converges up to subsequences to a  solution of
\begin{equation}
\label{limiteq}
\Delta w + |w|^{4 \over n-2} w = 0 \quad {\mbox {in}} \quad \R^n.
\end{equation}
Back to the original variable,  ``near $x_j$''
the behavior of $u_j (y)$ can be approximated as
\begin{equation}\label{pp0}
u_j (x) \sim {1 \over M_j} \, w \left( {x-x_j \over M_j^{q_j -1 \over 2}} \right)
\end{equation}
If the solution $u_j$ develops a {\it positive} bubbling around $x_j$, then the limit profile \equ{limiteq} is necessarily positive.
It is known,
see \cite{CGS,obata}, that for the convenient choice
$\alpha_n = (n(n-2))^{n-2\over 4}$, this
solution is explicitly given by
\begin{equation}
\bar w(z)= \alpha_n \left( { 1 \over { 1 + |z|^2}}\right)^{n-2\over 2} \kern -5pt .
\label{expsol}\end{equation}
which corresponds precisely to an extremal of $\sigma_n$,
the best constant in the critical Sobolev embedding,
\begin{equation}
\sigma_n  \, = \, \inf_{ u\in C_0^1(\R^n )\setminus\{0\}} \,
{\int_{\R^n}    |\nabla u|^2
\over
(\, \int_{\R^n}    |u|^{2n\over n-2}  )^{n-2\over n} }  ,
\label{SN}\end{equation} see
\cite{aubin,talenti}.
Thus, a solution blowing-up positively  near $x_j$ looks at main order as
\begin{equation}
u_j (x) \sim \alpha_n \left ( { 1 \over { 1+ M_j^{4\over n-2}\, |x-x_j|^2
}}\right )^{n-2\over 2}\kern -5pt M_j .
\label{pp}\end{equation}
In \cite{dfm,dfm1} this issue has been addressed
for a class of domains which includes that considered by Coron
in \cite{Co}.
It is established that a {\it positive}  solution to \equ{Q}
exists for $q=\frac{n+2}{n-2}+\ve $ with any small $\ve>0$, or equivalently to \equ{P},
if for instance
$\Omega $ is a smooth domain exhibiting a sufficiently small hole: considering $\ve$ as a
small parameter, the solution  exhibits single-bubbling
around exactly two points
and ceases to exist when $\ve=0$.
More precisely,
let $ {\mathcal D}$ be a bounded, smooth domain in $\R^n$, $n \ge 3$,
and $P$ a point of ${\mathcal D}$. Let us consider
the domain
\begin{equation} \Omega = {\mathcal D} \setminus
\overline B(P ,\delta ) \label{om}\end{equation} where $\delta >0$ is a small number.
Then
there exists a $\delta_0 >0$, which depends on $ {\mathcal D}$ and the point
$P$
 such that if $0< \delta <\delta_0 $ is fixed and $\Omega$ is the domain given
 by \equ{om},
 then the following holds:
There exists $\ve_0 >0$ and a solution $u_\ve$, $0<\ve<
\ve_0$ of \equ{P}  of the form
\begin{equation}
u_\ve (x) =
 \sum_{j=1}^2 \alpha_n \left (
 { 1 \over
 1\, +\,  \ve^{- \frac 2 {n-2}}\Lambda_{j\ve }^{-2}
\, |x-\xi_{j}^\ve |^2      }\right )^{n-2\over 2}  \Lambda_{j\ve }^{n-2\over 2}
\ve^{1\over 2}\, (\, 1 \, + \, o(1)\, ),
\label{exp}
\end{equation}
where
$o(1) \to 0 $ uniformly as $\ve\to 0$.
The numbers $\Lambda_{j\ve} $ and the points $\xi_{j}^\ve$ converge (up to subsequences)
to a
critical point of certain function built upon the Green's function
of $\Omega$.

Another kind of construction for positive solutions to \equ{P} has been recently proposed by Vaira \cite{vaira}: if $\Omega$ is such that Problem \equ{P} at $\ve =0$ admits a positive non-degenerate
solution $u_0$, then Problem \equ{P} has a solution that at main order looks like the sum of $u_0$ and a blowing-up profile as the one described in \equ{pp}.

Not much is known about sign-changing solutions to \equ{P}, in fact as far as we know no existence results are available in the literature. One may ask for existence of sign-changing solutions with a blowing-up profile  like the one described in \equ{pp},  with a minus sign in front,
namely sign-changing solution  blowing-up {\it negatively} at one or more points in $\Omega$.
Unfortunately,
sign-changing solutions blowing-up  {\it negatively} at one point or at two points do not exist, as shown
in \cite{benbou}.

The purpose of this work is to give the first construction of blowing-up sign-changing solutions for \equ{P}. In fact, we show the existence of solutions with the shape described in \equ{pp0},
where $w$ is a sign-changing solution to the limit profile \equ{limiteq}. Not much is known about sign-changing solutions  to \equ{limiteq}, being the only available results   \cite{D,dmpp1,dmpp2}.  Furthermore, in order to perform a gluing construction as the one described for positive solutions, an important property of the solution $w$ to the limit problem \equ{limiteq} is needed: its non-degeneracy.

In  \cite{dmpp1} it is proven that  there exists an integer $k_0$ such that for any integer $k \geq k_0$, a solution
solution $Q = Q_k$ to Problem
\begin{equation}
\label{eqqq}
\Delta u + |u|^{p-1} u  = 0 \quad {\mbox {in}} \quad \R^n, \quad p={n+2 \over n-2},
\end{equation}
exists. Furthermore, if we define  the energy by
\begin{equation} \label{energy1}
E(u) = {1\over 2} \int_{\R^n} |\nabla u|^2 \, dx - {1\over p+1} \int_{\R^n} |u|^{p+1} \, dx ,
\end{equation}
we have
$$
E (Q_k ) =  \left\{ \begin{matrix}   (k+1) \, S_n \, \left( 1+ O(k^{2-n})  \right)  & \hbox{ if } n\geq 4\, , \\ & \\     (k+1) \, S_3 \, \left( 1+ O(k^{-1} |\log k|^{-1} \right) & \hbox{ if } n=3 \end{matrix}
  \right.
$$
as $k \to \infty$,
where $S_n$ is a positive constant, depending on $n$. The solution $Q=Q_k$ decays at infinity like the fundamental solution, namely
\begin{equation}
\label{vanc1}
\lim_{|x| \to \infty} |x|^{n-2} \, Q_k (x ) = \left[ {4\over n(n-2) } \right]^{n-2 \over 4} \, 2^{n-2 \over 2} \left( 1+ c_k \right)
\end{equation}
where
$$
 c_k = \left\{ \begin{matrix}   O(k^{-1} )  & \hbox{ if } n\geq 4\, , \\ & \\     O(k^{-1} |\log k|^2 )  & \hbox{ if } n=3 \end{matrix}
  \right. \ \quad {\mbox {as}} \quad k \to \infty.
$$
Furthermore, the solution $Q=Q_k$ has a positive global non degenerate maximum at $x=0$. To be more precisely we have
\begin{equation}
\label{at0}
Q(x) = \left[ n (n-2) \right]^{n-2 \over 4} \, \left ( 1- {n-2 \over 2} |x|^2 + O(|x|^3 ) \right) \quad {\mbox {as}} \quad |x|\to 0,
\end{equation}
and also there exists $\eta >0$, depending on $k_0$, but independent of $k$,  so that
\begin{equation}
\label{at00} \eta \leq Q(x) \leq Q(0) \quad {\mbox {for all}} \quad |x|\leq {1\over 2},
\end{equation}
for any $k$.
Another property for  the solution $Q=Q_k$ is that it is invariant under rotation of angle ${2\pi \over k}$ in the $x_1 , x_2$ plane, namely
\begin{equation}
\label{sim00}
Q( e^{2\pi \over k} \bar x , x' ) = Q(\bar x , x' ), \quad \bar x= (x_1, x_3) , \quad x'= (x_3, \ldots , x_n ).
\end{equation}
It is even in the $x_j$-coordinates, for any $j=2, \ldots , n$
\begin{equation}
 Q (x_1,\ldots,x_j, \ldots, x_{n}  ) =  Q (x_1,\ldots,-x_j, \ldots, x_{n}  ),\quad  j=2,\ldots,n.
\label{sim22}\end{equation}
It respects invariance under Kelvin's transform:
\begin{equation}
 Q (x )\  = \  |x|^{2-n} Q (|x|^{-2} x)\ .
\label{sim33}\end{equation}
A detailed description of these solutions is given in Appendix \ref{A}. These solutions are non-degenerate, as proved in \cite{mw}, in the sense precised in Section \ref{des}. More precisely the dimensional of the kernels of the linearized operator at $Q$
$$ -\Delta \phi = p |Q|^{p-1} \phi $$
 is shown to be $3n$.

\medskip
In this paper we show that, if the domain $\Omega$ has a small hole, like in \cite{dfm}, then a large number of sign-changing solutions to \equ{P} exist: they blow-up with a profile $Q$ near two points of the domain, and they converges to $0$, as $\ve \to 0$, far from these two points.

\medskip
We have the validity of the following result

\begin{theorem}
\label{teo} Let ${\mathcal D}$ be a smooth bounded connected domain in $\R^n$ containing the origin $0$. There exists $\delta_0$ such that, if $\delta \in (0, \delta_0 )$ is fixed
and $\Omega$ is the set defined by
 $\Omega = {\mathcal D} \setminus \omega$, for any  smooth domain
$\omega \subset \bar B (0, \delta ) \subset {\mathcal D}$, then there exists a sequence
of $\ve_k$ such that, for any $\ve \in (0, \ve_k )$ there exists a sign changing solution $u_\ve$ to \equ{P}, with
$$
u_\ve (x ) = \sum_{j=1}^2 Q\left( {x-\xi_{j \ve} \over \Lambda_{j\ve}^2 \ve^{2\over n-2}} \right) \, \Lambda_{j\ve}^{n-2 \over 2} \ve^{1\over 2} \, \left( 1+ o(1) \right)
$$
where
$o(1) \to 0 $ uniformly as $\ve\to 0$. Up to subsequence,
$$\lim_{\ve \to 0} \Lambda_{j \ve} = \Lambda_j \in (0,\infty), \quad
\lim_{\ve \to 0} \xi_{j\ve } = \xi_j \in \Omega, \quad {\mbox {with}} \quad
\xi_1 \not= \xi_2 .
$$
Furthermore
$$
{1\over 2} \int_{\Omega } |\nabla u_\ve |^2 \, dx -{n-2 \over 2n +\ve (n-2) } \int_{\Omega }
|u_\ve |^{ {2n \over n-2} +\ve } = 2\, (k+1) \, S_n + O(\ve |\ln \ve |).
$$
\end{theorem}

\medskip
Theorem \ref{teo} exhibits new concentration phenomena in which  the basic cell of bubbling is not the positive solution. In the positive bubbling case, the kernel at the basic cell always contains $n+1$ dimensional kernels which corresponds to precisely the translation and scaling. For sign-changing bubbling solution, which is non-radial, the kernel at the basic cell contains not only the translations and scaling but also effects of rotation and Kelvin transform. The main difficulty is to find enough parameters to adjust. This is the main contributions of this paper. As we discover, the dominating role played is still the translation and scaling.

\medskip
We mention related results on sign-changing bubbling solutions. For scalar curvature type equations, examples of sign-changing blowing-up solutions are constructed by Robert-Vetois  in \cite{RV1}-\cite{RV2}. There negative bubbles are added to a positive solution to the Yamabe problem. The basic cell is still the single positive solution. Sign-changing bubbling solutions are constructed in the slightly subcritical problems ($\epsilon <0$) by Pistoia and Vetois \cite{PV}. The basic cell is a combination of positive and negative solution.

\medskip
In Section \ref{des} we precise the notion of {\it non-degeneracy} for $Q$. In Section \ref{des1} we describe the projection of the function $Q$ into $H^1_0 (\Omega )$ and
we give the expansion of the energy associated to the sum of two projected copies of $Q$.
Section \ref{scheme} is devoted to explain the construction of our solution and the scheme of the proof.

\section{About the non-degeneracy  of the basic cell}\label{des}

In \cite{mw}, we proved that these solutions are {\it non degenerate}.
To explain this, let us fix one solution $Q=Q_k$ of the family and define  the linearized equation around $Q$ for Problem \equ{eqqq} as follows
\begin{equation}
\label{defL}
L(\phi ) = \Delta \phi + p |Q|^{p-1}  \phi .
\end{equation}
The invariances \equ{sim00}, \equ{sim22}, \equ{sim33}, together with the natural invariance of any solution
to \equ{eqqq} under translation (if $u$ solves \equ{eqqq} then also $u (x+\xi)$ solves \equ{eqqq} for any $\xi \in \R^n$) and under dilation (if $u$ solves \equ{eqqq} then $\la^{-{n-2 \over 2} } u(\la^{-1} x)$ solves \equ{eqqq} for any $\la >0$) produce some {\it natural} functions $\varphi$  in the kernel of $L$, namely
$$
L(\varphi ) = 0.
$$
These are the $3n$  linearly independent functions we introduce next:
\begin{equation}
\label{capitalzeta0}
z_0 (x) = {n-2 \over 2} Q(x) + \nabla Q (x) \cdot x ,
\end{equation}
\begin{equation}
\label{capitalzetaj}
z_\alpha (x)  = {\partial \over \partial x_\alpha  } Q(x) , \quad {\mbox {for}} \quad \alpha=1, \ldots , n,
\end{equation}
and
\begin{equation}
\label{capitalzeta2}
z_{n+1} (x) = -x_2  {\partial \over \partial x_1 } Q(x) + x_1  {\partial \over \partial x_2  } Q(x),
\end{equation}
\begin{equation}
\label{chico1}
z_{n+2} (x) = -2 x_1 z_0 (x) + |x|^2 z_1 (x) , \quad z_{n+3} (x) =  -2 x_2 z_0 (x) + |x|^2 z_2 (x)
\end{equation}
and, for $l=3, \ldots , n$
\begin{equation}
\label{chico2}
z_{n+l+1} (x) = -x_l z_1 (x) + x_1 z_l (x), \quad z_{2n+l-1} (x) =  -x_l z_2 (x) + x_2 z_l (x).
\end{equation}
Indeed, a direct computation gives that
$$
L(z_\alpha ) = 0 , \quad {\mbox {for all}} \quad \alpha = 0 , 1 , \ldots , 3n-1.
$$
The function $z_0$ defined in \equ{capitalzeta0} is related to the invariance of Problem \equ{eqqq} with respect to dilation $\la^{-{n-2 \over 2} } Q (\la^{-1} x)$. The functions $z_i$, $i=1, \ldots , n$,  defined in \equ{capitalzetaj} are related to the invariance of Problem \equ{eqqq} with respect to translation $Q (x+\xi )$.
The function $z_{n+1}$ defined in \equ{capitalzeta2} is related to the invariance of $Q$ under rotation in the
$(x_1 , x_2)$ plane. The two  functions $z_{n+2}$ and $z_{n+3}$ defined in \equ{chico1} are related to the invariance of Problem \equ{eqqq} under Kelvin transformation \equ{sim33}. The functions defined in \equ{chico2} are related to the invariance under rotation in
the $(x_1 , x_l)$ plane and in the $(x_2 , x_l)$ plane respectively.

Let us be more precise. Denote by $O(n)$ the orthogonal group of $n\times n$  matrices $M$ with real coefficients, so that $M^T M = I$, and
by $SO(n)  \subset O(n) $ the special orthogonal group of all matrices in $O(n)$ with $det M=1$. $SO(n)$ is the group of all rotations in $\R^n$, it is a compact
group, which can be identified with a compact set in $\R^{n (n-1) \over 2}$.
Consider the sub group  $\hat S$ of $SO(n)$ generated by  rotations in the $(x_1 , x_2)$-plane, in the $(x_j , x_\alpha)$-plane, for any $j=1,2$ and $\alpha = 3, \ldots , n$. We have that $\hat S$ is compact and can be identified with a compact manifold of dimension  $2n-3$, with no boundary.
In other words, there exists a smooth injective map
$\chi : \hat S \to \R^{{n(n-1) \over 2}}$ so that $\chi ( \hat S)$ is a compact manifold of dimension $2n-3$ with no boundary and $\chi^{-1} : \chi (\hat S ) \to \hat S$ is a smooth parametrization of $\hat S$  in a neighborhood of the Identity. Thus  we write
$$
\theta \in K =  \chi (\hat S ), \quad
R_\theta = \chi^{-1} (\theta )
$$
where $K$  a compact manifold of dimension $2n-3$ with no boundary and $R_\theta$
denotes a rotation in $\hat S$.

Let $A = (\la ,  \xi , a,  \theta ) \in \R^+ \times \R^n \times \R^2 \times \R^{2n-3}$,
and define
\begin{equation}
\label{defthetaA}
\Theta_A (x) = \la^{-{n-2 \over 2}} \left| \eta_{\la, \xi , a} (x)  \right|^{2-n} \,
Q \left( { R_\theta \left( {x-\xi \over \la }  - a |{x-\xi \over \la }|^2 \right) \over | \eta_{\la, \xi , a} (x)|^2 }\right),
\end{equation}
where
\begin{equation}
\label{defetaA}
\eta_{\la, \xi , a} (x) = {x-\xi \over |x-\xi |} - a {|x-\xi | \over \la }
\end{equation}
and $Q$ is our fixed non degenerate solution to Problem \equ{eqqq} described above.
In \cite{DKM} it is proven that for  any choice of $A$, the function $\Theta_A$ is still a solution of \equ{eqqq}, namely
$$
\Delta \Theta_A + |\Theta_A |^{p-1} \Theta_A = 0 , \quad {\mbox {in}} \quad \R^n.
$$
For any set of parameters $A = (\la ,  \xi , a,  \theta ) \in \R^+ \times \R^n \times \R^2 \times \R^{2n-3}$, we introduce the function
\begin{eqnarray}
\label{va1ok}
Q_A (x) & = & \Theta_A (R_\theta^{-1} x) \nonumber \\
& = &   \la^{-{n-2 \over 2}} \left| \eta_{\la, \xi , a} ( R_\theta^{-1} x )  \right|^{2-n} \,
Q \left( { R_\theta \left( {R_\theta^{-1} x -\xi \over \la } - a |{R_\theta^{-1} x-\xi \over \la }|^2 \right) \over | \eta_{\la, \xi , a} (R_\theta^{-1} x)|^2 }\right)
\end{eqnarray}
More explicitly
$$
Q_A (x) = \la^{-{n-2 \over 2}} \left| {x- R_\theta \xi \over |x- R_\theta \xi |} - R_\theta a {|x- R_\theta \xi | \over \la } \right|^{2-n} \, \times
$$
$$
\times
Q \left( {  {x-R_\theta \xi \over \la } - R_\theta a |{x- R_\theta \xi \over \la }|^2 \over 1-2 R_\theta a \cdot ({x-R_\theta \xi \over \la } ) + |a|^2 |{x-R_\theta \xi \over \la }|^2 }\right).
$$
Easy but long computations give the following {\it natural} relations between $z_\alpha$ and  differentiation of $Q_A$ with respect to each component of $A$. More precisely, one has
\begin{equation}
\label{z011}
z_0 (y) = - \, {\partial \over \partial \la} \left[ Q_A (x) \right]_{| \la = 1, \xi = 0, a =0 , \theta=0 }
\end{equation}
\begin{equation}
\label{zj11}
z_\alpha  (y) = - \, {\partial \over \partial \xi_\alpha } \left[ Q_A (x) \right]_{| \la = 1, \xi = 0, a =0 , \theta=0 },
\quad \alpha = 1, \ldots , n,
\end{equation}
 \begin{equation}
\label{z211}
z_{n+2} (y) =  \, {\partial \over \partial a_1} \left[ Q_A (x) \right]_{| \la = 1, \xi = 0, a =0 , \theta=0 }
\end{equation}
 \begin{equation}
\label{z311}
z_{n+3} (y) =  \, {\partial \over \partial a_2} \left[ Q_A (x) \right]_{| \la = 1, \xi = 0, a =0 , \theta=0 }.
\end{equation}
Now, let $\theta = (\theta_{12}, \theta_{13}, \ldots , \theta_{1n}, \theta_{23} , \ldots , \theta_{2n} ) $, where
$\theta_{ij}$ represents the rotation in the $(i,j)$-plane. Then
 \begin{equation}
\label{z411}
z_{n+1} (y) =  \, {\partial \over \partial \theta_{12}} \left[ Q_A (x) \right]_{| \la = 1, \xi = 0, a =0 , \theta=0 }
\end{equation}
and, for any $l=3, \ldots , n$,
 \begin{equation}
\label{z511}
z_{n+l+1} (y) =  \, {\partial \over \partial \theta_{1l}} \left[ Q_A (x) \right]_{| \la = 1, \xi = 0, a =0 , \theta=0 }
\end{equation}
 \begin{equation}
\label{z511}
z_{2n+l-1} (y) =  \, {\partial \over \partial \theta_{2l}} \left[ Q_A (x) \right]_{| \la = 1, \xi = 0, a =0 , \theta=0 }.
\end{equation}

Following \cite{DKM}, a solution $Q$ is said to be non degenerate if
\begin{equation}
\label{nondeg} {\mbox {Kernel}} (L) = {\mbox {Span}} \{ z_\alpha \, : \, \alpha = 0 , 1 , 2, \ldots , 3n-1 \},
\end{equation}
or equivalently, any bounded (or any solution in ${\mathcal D}^{1,2}$) of $L(\varphi ) = 0 $ is a linear combination of the functions $z_\alpha $, $\alpha = 0 , \ldots , 3n-1$.

In \cite{mw} we proved that, under certain condition on the dimension $n$, the solution $Q$ is non-degenerate.
Indeed, in \cite{mw} we showed that in all dimensions $n\leq 48$, any solution $Q=Q_k$
is non degenerate in the sense defined above. If dimension $n \geq 49$, our result \cite{mw}  guarantees the existence of a subsequence of solutions $Q_{k_j}$ each one of which is non degenerate.

\section{First approximation and expansion of the energy}\label{des1}

The existence of a non degenerate solution in ${\mathcal D}^{1,2} (\R^n )$ of \equ{eqqq} is the basic element for our construction. In fact, we can perform our construction starting from any one of the infinitely many solutions $Q=Q_k$
with the property of being non degenerate.
Thus, from now on,  we  fix a function $Q=Q_k$, and for simplicity of notation we drop the index $k$.

The function $Q_A$ defined in \equ{va1ok} will be the building block of our construction. We first correct it so that it satisfies zero boundary condition on $\partial \Omega$.
This is done defining  $P Q_A$ to be the projection of $Q_A$ onto $H^1_0 (\Omega)$, namely the  unique solution to
\begin{equation}
\label{va2} \Delta u = \Delta Q_A \quad {\mbox {in}} \quad \Omega, \quad u = 0 \quad {\mbox {on}} \quad
\partial \Omega.
\end{equation}
In other words, $P Q_A = Q_A - \varphi_A$ where $\varphi_A$ solves
\begin{equation}
\label{va3} \Delta u = 0 \quad {\mbox {in}} \quad \Omega, \quad u = Q_A \quad {\mbox {on}} \quad
\partial \Omega.
\end{equation}
Next Lemma provides a precise description of $P Q_A$ and $\varphi_A$, when $\la \to 0$.
To state the result we need to recall the following.  Let
us denote by $G(x,y)$ the Green's function of the domain, namely
$G$ satisfies \begin{equation}
\label{green} \Delta_x G(x,y) = \delta(x-y) ,\quad x\in \Omega
, \quad G (x,y) = 0 ,\quad x\in
\partial \Omega , \end{equation} where $\delta (x)$ denotes the Dirac mass at
the origin. We denote by $H(x,y)$ its regular part, namely \begin{equation}
\label{acca}
H(x,y) = \Gamma (x-y) -G(x,y)
\end{equation} where $\Gamma$ denotes the
fundamental solution of the Laplacian, \begin{equation}
\label{fonda} \Gamma (x) = b_n
|x|^{2-n} ,
\end{equation} so that $H$ satisfies $$ \Delta_x H(x,y) = 0 ,\quad
x\in \Omega , \quad H(x,y) = \Gamma (x-y) ,\quad x\in \partial
\Omega . $$ Its {\em diagonal} $H(x,x)$ is usually called Robin function.

\medskip

\begin{lemma} \label{expaPQ}
Assume that $a \in \R^2$, $\xi \in \R^n $, and $\theta \in K$ are fixed so that
 $R_\theta \xi \in \Omega$.
We have the validity of the following estimates:
\begin{equation}
\label{va3}
\varphi_A (x) = b_n^{-1} \, \la^{n-2 \over 2} \, Q \left(  R_\theta a \right) \, H(x, R_\theta \xi ) + O(\la^{n\over 2} )
\end{equation}
uniformly for $x \in \Omega$, as $\la \to 0$. Furthermore,
\begin{equation}
\label{va4}
PQ_A (x) = b_n^{-1} \, \la^{n-2 \over 2} \, Q  \left(  R_\theta a \right) \, G(x, R_\theta \xi ) + O(\la^{n\over 2} )
\end{equation}
uniformly for $x$ in compact sets of $\Omega \setminus \{ R_\theta \xi \}$, as $\la \to 0$. In \equ{va3} and \equ{va4}, $H$ and $G$ are the functions defined respectively in \equ{acca} and \equ{green}, $b_n$ is a positive constant defined in \equ{fonda}, and
$O(1)$ denotes a smooth function of $x$ which is uniformly bounded as $\la \to 0$.

\end{lemma}

\begin{proof} Let $x \in \partial \Omega$ and let $\delta >0 $ so that dist $(x, R_\theta \xi ) > \delta$.
Define
  $X = {x-R_\theta \xi \over \la }$. The argument of the function $Q$ in \equ{va1ok} gets re-written as
\begin{eqnarray*}
{  {x-R_\theta \xi \over \la } - R_\theta a |{x- R_\theta \xi \over \la }|^2  \over 1-2 R_\theta a \cdot ({x-R_\theta \xi \over \la } ) + |a|^2 |{x-R_\theta \xi \over \la }|^2 } & = & {  X - R_\theta a |X|^2  \over 1-2 R_\theta a \cdot X  + |a|^2 |X|^2 }
\end{eqnarray*}
 Observe that $X={x-R_\theta \xi \over \la} \to \infty $ uniformly on $x \in \partial \Omega$, as $\la \to 0$. Thus we have
\begin{eqnarray}\label{vancu1}
{  X - R_\theta a |X|^2  \over 1-2 R_\theta a \cdot (X ) + |a|^2 |X|^2 }   &=&
{  X- R_\theta a |X|^2  \over |a|^2 |X|^2  \left[  1-2 {R_\theta a \over |a|^2} \cdot {X \over |X|^2}  + {1\over |a|^2 |X|^2 }\right]^2 } \nonumber \\
&=& - {R_\theta a \over |a|^2} \left[ 1-2 {R_\theta a \over |a|^2} \cdot {X \over |X|^2} + {1\over |a|^2 |X|^2 } \right]^{-2} \nonumber \\
&+& {X \over |a|^2 |X|^2 } \left[ 1-2 {R_\theta a \over |a|^2} \cdot {X \over |X|^2} + {1\over |a|^2 |X|^2 } \right]^{-2}
\end{eqnarray}
A Taylor expansion gives that, uniformly for $x \in \partial \Omega$, one has
\begin{eqnarray*}
\left[ 1-2 {R_\theta a \over |a|^2} \cdot {X \over |X|^2} + {1\over |a|^2 |X|^2 } \right]^{-2}
&=& 1-2 \left( -2 {R_\theta a \over |a|^2} \cdot {X \over |X|^2} + {1\over |a|^2 |X|^2 } \right) \\
&+&
3 \left( -2 {R_\theta a \over |a|^2} \cdot {X \over |X|^2} + {1\over |a|^2 |X|^2 } \right)^2 + O(\la^2 )\\
&=& 1+ 4 {R_\theta a \over |a|^2} \cdot {X \over |X|^2}  \\
&-&  {2\over |a|^2 |X|^2 } + 12
\left| {R_\theta a \over |a|^2} \cdot {X \over |X|^2} \right|^2  + O(\la^3 ) \\
&=& 1+4 {R_\theta a \over |a|^2} \cdot {X \over |X|^2} + O(\la^2 )
\end{eqnarray*}
as $\la \to 0 $. Inserting this information in \equ{vancu1}, we get that
\begin{eqnarray}\label{vancu2}
{  X - R_\theta a |X|^2  \over 1-2 R_\theta a \cdot (X ) + |a|^2 |X|^2 }   &=&
- {R_\theta a \over |a|^2 } + {X \over |a|^2 |X|^2} \nonumber \\
&-&  4 {R_\theta a \over |a|^2} \cdot {X \over |X|^2}  + O(\la^2 ) \\
&=& - {R_\theta a \over |a|^2 }  + {x-R_\theta \xi \over  |a|^2 |x-R_\theta \xi |^2} \, \la \nonumber \\
&-& 4 {R_\theta a \over |a|^2} \, {R_\theta a \over |a|^2} \cdot {x-R_\theta \xi \over |x-R_\theta \xi|^2 } \, \la + O(\la^2).  \nonumber
\end{eqnarray}
Thanks to \equ{vancu2}, we get that
\begin{eqnarray}
\label{vancu3}
& & Q \left( {  X - R_\theta a |X|^2  \over 1-2 a \cdot (X ) + |a|^2 |X|^2 }  \right) =
Q (-{R_\theta a \over |a|^2} )  \\
&&+ \nabla Q (-{R_\theta a \over |a|^2} ) [{x-R_\theta \xi \over  |a|^2 |x-R_\theta \xi |^2}
-4 {R_\theta a \over |a|^2} \, {R_\theta a \over |a|^2} \cdot {x-R_\theta \xi \over |x-R_\theta \xi|^2 }  ] \la  + O(\la^2)\nonumber
\end{eqnarray}
uniformly for $x \in \partial \Omega$, as $\la \to 0$.

On the other hand, in the same region $x \in \partial \Omega$, we have that
\begin{eqnarray}
\label{vancu4}
\left| {x-R_\theta \xi \over |x-R_\theta \xi | } - R_\theta a { |x-R-\theta \xi | \over \la } \right|^{2-n} &=&
{\la^{n-2} \over |a|^{n-2} |x-R_\theta \xi |^{n-2} } \nonumber \\
& +& {n-2 \over |a|^{n-2} |x-R_\theta \xi|^{n-2}} \, {R_\theta a\over |a|^2} \cdot {x-R_\theta \xi \over |x-R_\theta \xi |^2} \, \la^{n-1} \\
&+& O(\la^n ). \nonumber
\end{eqnarray}
We conclude from \equ{vancu3} and \equ{vancu4} that, uniformly for $x \in \partial \Omega$,
\begin{eqnarray}
\label{vancu5}
\varphi_A (x) &=& {\lambda^{n-2 \over 2} \over |x-R_\theta \xi |^{n-2}} \, {1\over |a|^{n-2} }
Q(-{R_\theta a \over |a|^2 }) \nonumber \\
&+&  {\lambda^{n \over 2} \over |x-R_\theta \xi |^{n-2}} \, {1\over |a|^{n-2} }
Q(-{R_\theta a \over |a|^2 }) \, (n-2) \, {R_\theta a \over |a|^2} \cdot {x-R_\theta \xi \over |x-R_\theta \xi |^2} \nonumber \\
&+&  {\lambda^{n \over 2} \over |x-R_\theta \xi |^{n-2}} \, {1\over |a|^{n-2} }  [{x-R_\theta \xi \over  |a|^2 |x-R_\theta \xi |^2}
-4 {R_\theta a \over |a|^2} \, {R_\theta a \over |a|^2} \cdot {x-R_\theta \xi \over |x-R_\theta \xi|^2 }  ] \nonumber
\\
&+& O(\la^{n+2 \over 2} ) \nonumber \\
&=& {\lambda^{n-2 \over 2} \over |x-R_\theta \xi |^{n-2}} \,
Q(R_\theta a ) + O(\la^{n\over 2} )
\end{eqnarray}
as $\la \to 0$. Observe that we have used the fact that ${1\over |a|^{n-2} }
Q(-{R_\theta a \over |a|^2 })  =
Q(R_\theta a ) $.
A direct application of the maximum principle guarantees the validity of \equ{va3}.
To prove \equ{va4}, it is enough to observe that, for any $x $ in a compact set of $\Omega \setminus \{ R_\theta \xi \}$,
\begin{equation}
\label{ma}
Q_A (x) = {\la^{n-2 \over 2} \over |x-R_\theta \xi |^{n-2}}
\, Q(R_\theta a ) + O(\la^{n\over 2} )
\end{equation}
as $\la \to 0$. This concludes the proof of the Lemma.
\end{proof}

\medskip
We consider now two sets $A_1 = (\la_1 , \xi_1 , a_1 , \theta_1 )$ and
 $A_2 = (\la_2 , \xi_2 , a_2 , \theta_2 )$
 and the functions
 $$
Q_i = Q_{A_i}, \quad  PQ_i = P  Q_{A_i} , \quad i=1,2 .
 $$
 Our purpose is to estimate the following
 quantity
 $$
 J_0 (  PQ_1 + PQ_2 ) =
 {1\over 2} \int_\Omega
 |\nabla (  PQ_1 + PQ_2)|^2 -
 {1\over p +1 } \int_\Omega
 (  PQ_1 + PQ_2 )^{p+1} .
 $$
We use the notations
\begin{equation}\label{not}
\hat \xi_i =  R_{\theta_i } \xi_i , \quad \hat a_i = R_{\theta_i } a_i , \quad i=1,2.
\end{equation}
Let us now fix a number $\delta >0$ and consider the following constraints
\begin{eqnarray}
\label{constra}
 (\hat  \xi_1 , \hat \xi_2 ) \in \Omega \times \Omega  & : &
{\mbox {dist}} ( \hat  \xi_i , \partial \Omega ) > \delta , \, |\hat  \xi_1 - \hat \xi_2 | >\delta \nonumber  \\
\theta_i \in K, & & |a_i | \leq {1\over 2} , \quad i=1,2.
\end{eqnarray}
 Let us set
 \begin{equation}
\label{defcn}
 \gamma_n =
 {1\over 2} \int_{\R^n}
 |\nabla Q|^2 -
 {1\over p +1 } \int_{\R^n}
 |Q|^{p+1}
\end{equation}

\medskip
\begin{lemma}
Given $\delta >0$ we have the validity of the expansion
\begin{eqnarray}\label{expa0}
 J_0 (  PQ_1 + PQ_2 ) &=&
 2\gamma _n + \alpha_n
 H( \hat  \xi_1, \hat  \xi_1) \, Q ( \hat  a_1 )^2 \, \lambda_1^{n-2} \nonumber \\
&+&  \alpha_n
 H( \hat  \xi_2,  \hat  \xi_2) \, Q ( \hat  a_2 )^2 \, \lambda_2^{n-2} \nonumber \\
& -& \alpha_n
 2G( \hat  \xi_1,  \hat  \xi_2) \, Q ( \hat  a_1 )  Q ( \hat  a_2 ) \,
 \lambda_1^{n-2\over 2 }
 \lambda_2^{n-2\over 2 } \nonumber \\
& + &
 o(
 \max\{ \lambda_1, \lambda_2\} ^{n-2} )
 \end{eqnarray}
as $\lambda_1 , \lambda_2 \to 0$,
 uniformly in the set satisfying constraints \equ{constra}. In \equ{expa0}, $\alpha_n$ denotes a fixed positive constant, independent of $\ve$.
 \end{lemma}

\begin{proof}
The full expansion \equ{expa0} is consequence of the following
formulas:
 \begin{eqnarray}
 \int_\Omega
 |\nabla  PQ_i |^2 &=&
 \int_{\R^n} |\nabla  Q|^2  -
 2\alpha_n   H(\hat  \xi_i, \hat  \xi_i) \, Q^2 ( \hat  a_i ) \, \lambda_i^{n-2} \nonumber \\
&+&
 o(\lambda_i^{n-2} ),
 \label{a1}\end{eqnarray}
\begin{eqnarray}
 \int_\Omega
 \nabla  PQ_1
 \nabla  PQ_2
 &=&
 2\alpha_n
 G( \hat  \xi_1,  \hat  \xi_2) \, Q (  \hat  a_1 ) Q (  \hat  a_2 )  \,
 \lambda_1^{n-2\over 2 }
 \lambda_2^{n-2\over 2 } \nonumber \\
& + &
 o(
 \max\{ \lambda_1, \lambda_2\} ^{n-2} ) ,
 \label{a2}\end{eqnarray}
\begin{eqnarray}  \label{a6}
 {1\over p +1 } \int_\Omega
& & \left[ |PQ_1 + PQ_2|^{p+1} -
   |PQ_1| ^{p+1} -
   |PQ_2| ^{p+1} \right] \nonumber \\
&=& 4 \alpha_n
 G( \hat  \xi_1,  \hat  \xi_2) \, Q (  \hat  a_1 ) Q (  \hat  a_2 )  \,
 \lambda_1^{n-2\over 2 }
 \lambda_2^{n-2\over 2 }\nonumber  \\
&
  + &
 o(
 \max\{ \lambda_1, \lambda_2\} ^{n-2} )
\end{eqnarray}
 and
 \begin{eqnarray} {1\over p+1} \int_\Omega
   |PQ_i| ^{p+1}  &=&
{1\over p+1} \int_{\R^n}
 |Q|^{p+1}  \nonumber \\
&-&
  2\alpha_n  H(  \hat  \xi_i, \hat  \xi_i) \, Q^2 ( \hat  a_i )   \,
  \lambda_i^{n-2} \nonumber \\
&+&
  o(\lambda_i^{n-2} ) .
  \label{a7}\end{eqnarray}
Indeed, we decompose
\begin{eqnarray*}
 J_0 (  PQ_1 + PQ_2 ) &=&
 \sum_{i=1,2} {1\over 2}
\left[ \int_\Omega
 |\nabla  PQ_i |^2 - {1\over p+1} \int_\Omega  |PQ_i| ^{p+1}  \right] \\
&
 +&
 \int_\Omega
 \nabla  PQ_1
 \nabla  PQ_2
\\
&-&
 {1\over p +1 } \int_\Omega \left[
 |PQ_1 + PQ_2|^{p+1} -
  |PQ_1| ^{p+1} -
   |PQ_2| ^{p+1} \right].
\end{eqnarray*}
 Thus substituting  estimates \equ{a1}, \equ{a2}, \equ{a6} and \equ{a7}
 in this relation we obtain the thesis.

\medskip
In what is left of this proof, we shall show the validity of \equ{a1}, \equ{a2}, \equ{a6} and \equ{a7}.

\medskip
\noindent
{\it Proof of \equ{a1}.} For simplicity, we write $\varphi_i (x) = \varphi_{A_i } (x)$. An integration by parts
gives that
\begin{equation}
\label{uffa}
\int_{\Omega} |\nabla PQ_i|^2 \, dx = \int_{\Omega} |Q_i |^{p+1} \, dx -
\int_{\Omega} |Q_i|^{p-1} Q_i  \varphi_i \, dx
\end{equation}
taking into account that $- \Delta P Q_i = |Q_i|^{p-1} Q_i$.
Let us first compute $\int_{\Omega} |Q_i |^{p+1} \, dx$, writing
$$
\int_{\Omega} |Q_i |^{p+1} \, dx = \int_{|x- \hat \xi_i | < \delta} |Q_i |^{p+1} \, dx + \int_{|x-\hat \xi_i |> \delta} |Q_i |^{p+1} \, dx
$$
for some $\delta >0$ fixed and small. In the ball $ |x- \hat \xi_i | < \delta$ we introduce the change of variables $ y = {x- \hat \xi_i \over \la_i}$, so that
\begin{eqnarray*}
 \int_{|x- \hat \xi_i | < \delta} |Q_i |^{p+1} \, dx & = &
\int_{|y|<{\delta \over \la_i}} \left[ \left| {y \over |y|} -  \hat  a_i |y| \right|^{2-n}
|Q| \left( {{y\over |y|^2} -  \hat  a_i \over | {y \over |y|^2 } - \hat  a_i |^2 } \right) \right]^{p+1} dy
\\
&=& \int_{|y|<{\delta \over \la_i }} \left[ \left| {y \over |y|^2} -  \hat  a_i \right|^{2-n} |y|^{2-n}
|Q| \left( {{y\over |y|^2} -  \hat  a_i \over | {y \over |y|^2 } - \hat  a_i |^2 } \right) \right]^{p+1} dy
\\
& & {\mbox {since}} \quad |w|^{2-n} Q ({w\over |w|^2} ) = Q(w) \\
&=& \int_{|y|<{\delta \over \la_i}} \left[ |y|^{2-n}
|Q| \left( {y\over |y|^2} -  \hat  a_i \right) \right]^{p+1} dy
\\
& =&  \int_{|y|<{\delta \over \la_i }}  |y|^{2n} \left[
|Q| \left({y\over |y|^2} -  \hat  a_i  \right) \right]^{p+1} dy \\
& & {\mbox {using the change of variables}} \quad z= {y\over |y|^2}  \\
& =&  \int_{|z|>{\la_i \over \delta}}   \left| Q \left( z - \hat  a_i \right) \right|^{p+1} dz =  \int_{\R^n}   \left|Q \right|^{p+1} dz  + O(\la_i^n),
\end{eqnarray*}
where $O(1)$ denotes a generic smooth function of the parameters that is uniformly bounded as $\la_i \to 0$.
Observe that the last expansion is consequence of \equ{at0}, \equ{at00} and \equ{constra}.
On the other hand,  in the set ${|x- \hat  \xi_i |> \delta}$ we have the validity of the expansion \equ{ma}, so that we conclude that
$$
\left|  \int_{|x- \hat  \xi_i |> \delta} |Q_i |^{p+1} \, dx \right| \leq C \la_i^{n+2 \over 2} \, |Q (\hat  a_i ) | \leq C \la_i^{n+2 \over 2}
$$
where again we use the assumption in \equ{constra} that $|a_i | \leq {1\over 2} $, and also \equ{at0}-\equ{at00}.

We thus conclude that
\begin{equation}
\label{uffa0}
\int_{\Omega}   \left|P Q_i \right|^{p+1} dz  = \int_{\R^n}   \left|Q \right|^{p+1} dz  + O(\la_i^{n+2 \over 2} ).
\end{equation}
We turn now to $\int_{\Omega} |Q_i|^{p-1} Q_i \varphi_i \, dx$. We claim that
\begin{equation}
\label{uffa1}
\int_{\Omega} |Q_i|^{p-1} Q_i \varphi_i \, dx = b_n^{-2} \la_i^{n-2 } Q^2 (\hat  a_i ) H(  \hat  \xi_i ,  \hat  \xi_i )  + O(\la_i^{n-1} ),
\end{equation}
where $O(1)$ is uniformly bounded, as $\la_i \to 0$, in the set of parameters satisfying \equ{constra}.
We decompose
$$
\int_{\Omega} |Q_i|^{p-1} Q_i \varphi_i \, dx = \int_{|x- \hat  \xi_i |<\delta } |Q_i|^{p-1} Q_i \varphi_i \, dx + \int_{\Omega \cap |x-  \hat  \xi_i |>\delta  } |Q_i|^{p-1} Q_i \varphi_i \, dx
$$
Recalling the validity of the expansion \equ{va3} for $\varphi_i$, we get that
\begin{eqnarray*}
 \int_{|x-R_{\theta_i } \xi_i |<\delta } |Q_i|^{p-1} Q_i \varphi_i \, dx &=&
b_n^{-1} \la_i^{n-2 \over 2} Q (\hat  a_i ) H(  \hat  \xi_i ,  \hat  \xi_i )
\left(  \int_{|x- \hat  \xi_i |<\delta } |Q_i|^{p-1} Q_i \, dx  \right) \\
&+& \left(  \int_{|x-  \hat  \xi_i |<\delta } |Q_i|^{p-1} Q_i \, dx  \right)  O (\la_i^{n-1} ) \\
&+& \left(  \int_{|x- \hat   \xi_i |<\delta } |Q_i|^{p-1} Q_i \,  |x-\hat  \xi_i |dx  \right)
O (\la_i^{n-2 \over 2} ).
\end{eqnarray*}
 In the ball $ |x-  \hat  \xi_i | < \delta$ we introduce the change of variables $ y = {x-  \hat  \xi_i \over \la_i}$, and using the invariance of $Q$ under Kelvin transform,
\begin{eqnarray*}
 \int_{|x-  \hat \xi_i | < \delta} |Q_i |^{p-1} Q_i \, dx
&=& \la_i^{n-2 \over 2} \int_{|y|<{\delta \over \la}} \left[ |y|^{2-n}
|Q| \left( {y\over |y|^2} - \hat  a_i \right) \right]^{p} dy
\\
& =&  \la_i^{n-2 \over 2}\int_{|y|<{\delta \over \la}}  |y|^{n+2} \left[
|Q| \left({y\over |y|^2} -  \hat a_i  \right) \right]^{p} dy \\
& & {\mbox {using the change of variables}} \quad z= {y\over |y|^2}  \\
& =& \la_i^{n-2 \over 2}  \int_{|z|>{\la \over \delta}} {1\over |z|^{n-2}}  \left| Q \left( z - \hat  a_i \right) \right|^{p} dz  \\
&=& \la_i^{n-2 \over 2} \int_{\R^n}   {1\over |z+  \hat  a_i |^{n-2} } \left|Q \right|^{p} dz  + o(  \la_i^{n-2 \over 2}   ),
\end{eqnarray*}
as $\la_i \to 0$.
Recall now that
$$
 Q(\hat a_i ) = b_n \, \int_{\R^n}   {1\over |z+  \hat  a_i |^{n-2} } \left|Q \right|^{p} dz
$$
Thus we get
\begin{equation}
\label{inter}
 \int_{|x-  \hat \xi_i | < \delta} |Q_i |^{p-1} Q_i \, dx = b_n^{-1} \la_i^{n-2\over 2} Q (\hat a_i ) + o (\la_i^{n-2 \over 2} ).
\end{equation}
On the other hand, using again the change of variables $ y = {x-  \hat  \xi_i \over \la_i}$
one finds directly that
$$
 \int_{|x- \hat   \xi_i |<\delta } |Q_i|^{p-1} Q_i \,  |x-\hat  \xi_i |dx = O (\la_i^{n\over 2} ).
$$
On the other hand, we have
$$
 \int_{\Omega \cap |x-  \hat  \xi_i |>\delta  } |Q_i|^{p-1} Q_i \varphi_i \, dx = O(\la_i^{n-1} )
$$
as direct consequence of \equ{va3} and \equ{ma}. Collecting the above estimates we get the validity of \equ{uffa1}. Expansion \equ{a1} follows directly from \equ{uffa0} and \equ{uffa1}.

\medskip
\noindent
{\it Proof of \equ{a2}.} Arguing as in the proof of \equ{a1}, it holds
$$
 \int_\Omega
 \nabla  PQ_1
 \nabla  PQ_2 =  \int_\Omega
  |Q_1 |^{p-1} Q_1
 PQ_2 = \left( \int_{|x-\hat \xi_1 | < \delta}
  |Q_1 |^{p-1} Q_1
 PQ_2  \, dx \right)  \, (1+ o(1) ),
$$
as $\la_i \to 0$, $i=1,2$. Now, using \equ{va4}, we get
\begin{eqnarray}
\label{inter1}
  \int_{|x-\hat \xi_1 | < \delta}
  |Q_1 |^{p-1} Q_1
 PQ_2  \, dx  &=& b_n^{-1} \la_2^{n-2 \over 2} Q (\hat a_2) G(\hat \xi_1 , \hat \xi_2 )  \times \nonumber \\
& &
 \left( \int_{|x-\hat \xi_1 | < \delta}
  |Q_1 |^{p-1} Q_1
   \, dx \right)  (1+ O(\la_2 ) ) \nonumber
\end{eqnarray}
Thanks to \equ{inter}, we conclude that
\begin{eqnarray}
\label{inter1}
  \int_{|x-\hat \xi_1 | < \delta}
  |Q_1 |^{p-1} Q_1
 PQ_2  \, dx  &=& b_n^{-2} \la_1^{n-2 \over 2} \la_2^{n-2 \over 2} Q (\hat a_1 ) Q (\hat a_2) G(\hat \xi_1 , \hat \xi_2 )   \nonumber \\
&+ &  o(
 \max\{ \lambda_1, \lambda_2\} ^{n-2} ) .
\end{eqnarray}
This gives the validity of \equ{a2}.

\medskip
\noindent
{\it Proof of \equ{a6}.} We write
\begin{eqnarray*}
 {1\over p +1 } \int_\Omega & &
 \left[ |PQ_1 + PQ_2|^{p+1} -
   |PQ_1| ^{p+1} -
   |PQ_2| ^{p+1} \right]
\\
&=& \int_{|x-\hat \xi_1 |<\delta } |PQ_1|^{p-1} PQ_1 P Q_2 \\
& + & \int_{|x-\hat \xi_2 |<\delta } |P Q_2|^{p-1} P Q_2  P Q_1 + O( (\la_1 \la_2 )^{n+2 \over 2}) \\
&=& \int_{|x-\hat \xi_1 |<\delta } |Q_1|^{p-1} Q_1 P Q_2 \\
& + & \int_{|x-\hat \xi_2 |<\delta } |Q_2|^{p-1} Q_2  P Q_1 + O( (\la_1 \la_2 )^{n+2 \over 2})
\end{eqnarray*}
At this point, \equ{a6} follows directly from \equ{inter1}.

\medskip
\noindent
{\it Proof of \equ{a7}.} Using the estimate \equ{ma}, we see that
$$
{1\over p+1} \int_{\Omega} |PQ_i |^{p+1} \, dx = {1\over p+1}
\int_{|x-\hat \xi_i |<\delta } |P Q_i |^{p+1} \, dx + O(\la_i^{n+2} ).
$$
On the other hand, a Taylor expansion gives
\begin{eqnarray*}
{1\over p+1}
\int_{|x-\hat \xi_i |<\delta } |P Q_i |^{p+1} \, dx & =&
{1\over p+1}
\int_{|x-\hat \xi_i |<\delta } |PQ_i |^{p+1} \, dx \\
& +&
\left( \int_{|x-\hat \xi_i |<\delta } | PQ_i |^{p-1} P Q_i \varphi_i \, dx \right) (1+ o(1))\\
&=&{1\over p+1}
\int_{|x-\hat \xi_i |<\delta } |PQ_i |^{p+1} \, dx \\
& +&
\left( \int_{|x-\hat \xi_i |<\delta } | Q_i |^{p-1}  Q_i \varphi_i \, dx \right) (1+ o(1))
\end{eqnarray*}
We now apply \equ{uffa0} and \equ{uffa1} to get \equ{a7}.
\end{proof}

\bigskip

We shall now choose the numbers
 $\lambda_i$ in terms of
$\ve$: we will assume \be \lambda_i^{n-2}  =  {\int_{\R^n} |Q|^{p+1} \over n \alpha_n }  \Lambda_i^2 \ve .
\label{L} \ee
Let us consider the energy functional, associated to problem
\equ{P},
\begin{equation}
\label{Jeps}
J_\ve ( u ) =
  {1\over 2} \int_\Omega
 |\nabla u |^2 -
 {1 \over p +1 + \ve  } \int_\Omega
u^{p+1 + \ve } .
\end{equation}
Arguing like in \cite{dfm}, we show that
 \begin{eqnarray*}
 J_\ve ( P Q_1 + P  Q_2 ) &=&
 J_0 (  P Q _1 +  P Q_2 ) \\
& +&
2\ve \left[ {1\over (p+1)^2} \int_{\R^n}  |Q|^{p+1} -
   {1\over p+1}\int_{\R^n}  |Q|^{p+1} \log  |U| \right]  \\
   &+&
   {(n-2)^2\over 4n } \ve \log (\lambda_1 \lambda_2) \int_{\R^n}
   |Q|^{p+1}
   + o(\ve) .
 \end{eqnarray*}
 Combining this estimate with the previous lemma, and our choice
\equ{L} for $\la_i$,
 we get the following result.

\begin{lemma}\label{paris}
Let $\delta >0$ and assume that
\begin{eqnarray}
\label{constra1}
 (\hat  \xi_1 , \hat \xi_2 ) \in \Omega \times \Omega  & : &
{\mbox {dist}} ( \hat  \xi_i , \partial \Omega ) > \delta , \, |\hat  \xi_1 - \hat \xi_2 | >\delta \nonumber  \\
\delta < \Lambda_i < \delta^{-1}, \quad & &
\theta_i \in K, \quad  |a_i | \leq {1\over 2} , \quad i=1,2,  \nonumber
\end{eqnarray}
where
$$
\lambda_i^{n-2}  =  {\int_{\R^n} |Q|^{p+1} \over n \alpha_n }  \Lambda_i^2 \ve , \quad \hat \xi_i =  R_{\theta_i } \xi_i , \quad \hat a_i = R_{\theta_i } a_i , \quad i=1,2.
$$
Then there exists $\ve_0 >0$ such that, for any $\ve \in (0,\ve_0 )$
 we have
\begin{eqnarray}\label{elis}
J_\ve ( U_1 +  U_2 ) &=&
 2\gamma_n + \chi_n \ve\log\ve + \eta_n \ve \\
&+&  \ve \, {\int_{\R^n} |Q|^{p+1} \over n} \, \Psi (A_1 , A_2 )   + o(\ve ),  \nonumber
\end{eqnarray}
 uniformly with respect to sets of parameters $A_1$ and $A_2$ satisfying \equ{constra1}.
Here \begin{eqnarray}  \Psi (A_1 , A_2)  &=&
 {1 \over 2}  H(\hat \xi_1, \hat \xi_1) Q^2 (\hat a_1 ) \Lambda_1^{2} +
 {1\over 2} H(\hat \xi_2,\hat  \xi_2) Q^2 (\hat a_2 ) \Lambda_2^{2}  \nonumber \\
&-&
 G(\hat \xi_1, \hat \xi_2) Q(\hat a_1 ) Q(\hat a_2 )
 \Lambda_1 \Lambda_2 +
 \log \Lambda_1\Lambda_2 ,
 \label{psi1}\end{eqnarray}
while $\eta_n$ and $\chi_n$ are the constants defined by
\begin{eqnarray*}
 \eta_n &=&2 \{ {1\over (p+1)^2} \int_{\R^n} |Q|^{p+1} -
   {1\over p+1}\int_{\R^n} |Q|^{p+1} \log |Q|\} \\
& +&
    {1\over p+1}
\chi_n \log \left[  {\int_{\R^n} |Q|^{p+1} \over n \alpha_n }  \right]
\end{eqnarray*}
and $\chi_n = {1\over p+1}
 \int_{\R^n} |Q|^{p+1}$.
\end{lemma}

\section{Scheme of the proof}\label{scheme}

The solution predicted by Theorem \ref{teo} will have  the form
\begin{equation}
\label{soltutta}
u(x)  = (1+ \zeta ) \, \left( u_0 (x) + \tilde \phi (x) \right),
\end{equation}
where $\zeta $ is the number defined by
\begin{equation}
\label{defzetamm}
\zeta = \ve^{ \ve \, {p \over 2 (p-1+\ve)} } -1
\end{equation}
and $u_0$ is the function defined by
\begin{equation}
\label{defu0}
 u_0 (x) =  PQ_1 (x) + P Q_2 (x)
\end{equation}
where for $j=1,2$, $PQ_j$ is the $H_0^1 (\Omega )$-projection of the function $Q_{A_j}$ defined in \equ{va1ok}, for some parameter $A_j = (\la_j , \xi_j , a_j , \theta_j )$. In \equ{soltutta}, the function $\tilde \phi $ has to be determined to have that $u$ is a solution to \equ{P}.

It is useful to rephrase Problem \equ{P} in some expanded domain
 $\Omega_\ve = \ve^{-{1\over n-2}}\, \Omega $. After the change of variables
\begin{equation}
\label{changeex}
v(y) = \ve^{1\over 2}   u (\ve^{1\over n-2} y ) , \quad y = \ve^{-{1\over n-2}} \, x \in \Omega_\ve,
\end{equation}
 Problem \equ{P} gets re-written as
\begin{equation}
\label{eqex}
\Delta v + v^{p+\ve} = 0 , \quad v>0 , \quad {\mbox {in}} \quad \Omega_\ve
\quad v=0 , \quad {\mbox {on}} \quad \partial \Omega_\ve,
\end{equation}
In the expanded variables, the solution in \equ{soltutta} gets into the form
\begin{equation}
\label{soltuttaex}
v(y) = v_0 (y) + \phi (y) , \quad {\mbox {where}} \quad v_0 (y) = \ve^{1\over 2}  u_0 (\ve^{1\over n-2} y ).
\end{equation}
In order to determine the unknown function $\phi$, we proceed in two steps. In the first step, we fix the parameters $A_1$ and $A_2$, and we find $\phi$ as solution of a proper non linear projected problem.
With abuse of notation, we denote
$$
A_j = (\Lambda_j , \xi_j , a_j , \theta_j ) \in \R_+ \times \Omega \times \R^2 \times K
$$
and
we assume the following constraints on $A_j$, $j=1,2$:
\begin{eqnarray}
\label{constra1}
 (\hat  \xi_1 , \hat \xi_2 ) \in \Omega \times \Omega  & : &
{\mbox {dist}} ( \hat  \xi_i , \partial \Omega ) > \delta , \, |\hat  \xi_1 - \hat \xi_2 | >\delta \nonumber  \\
& & \\
\delta < \Lambda_i < \delta^{-1}, \quad & &
\theta_i \in \chi (\hat S ), \quad  |a_i | \leq {1\over 2} , \quad i=1,2,  \nonumber
\end{eqnarray}
for some fixed $\delta >0$, where
$$
\lambda_i^{n-2}  = \beta_n  \Lambda_i^2 \ve , \quad \hat \xi_i =  R_{\theta_i } \xi_i , \quad \hat a_i = R_{\theta_i } a_i , \quad i=1,2,
$$
with $\beta_n $ the positive number defined by $\beta_n =  {\int_{\R^n} |Q|^{p+1} \over n \alpha_n }$.
Observe that a direct computation gives that, as $\ve \to 0$,
\begin{eqnarray} \label{defv0}
v_0 (y)&=& \sum_{j=1}^2  (\Lambda_j^*)^{-{n-2 \over 2}} \left| \eta_{\Lambda_j^* ,  R_{\theta_j} \xi_j' , a_j } (y ) \right|^{2-n} Q \left(   { {y-R_{\theta_j} \xi_j' \over \Lambda_j^* } - R_{\theta_j} a_j |{y- R_{\theta_j } \xi_j \over \Lambda_j^*} |^2  \over \eta_{\Lambda_j^* ,  R_{\theta_j} \xi_j' a_j } (y ) } \right) \nonumber \\
&+& o(1)
\end{eqnarray}
where  the function $\eta$ is defined in \equ{defetaA},  $\Lambda_i^{*} = ( \beta_n \Lambda_i^2 )^{1 \over {n-2}} $, see \equ{L}, and $\xi_j' = \ve^{-{1\over n-2}} \xi_j $.

 Consider  the
following functions, for any $\alpha = 0, 1, \ldots , 3n-1$, and $j=1,2$
$$ \bar Z_{\alpha j} = (\Lambda_j^*)^{-{n-2 \over 2}} \left| \eta_{\Lambda_j^* , R_{\theta_j} \xi_j' ,a_j   } (y ) \right|^{2-n} z_{\alpha} \left(  { {y-R_{\theta_j} \xi_j' \over \Lambda_j^* } - R_{\theta_j} a_j |{y- R_{\theta_j} \xi_j' \over \Lambda_j^* }|^2 \over  \eta_{\Lambda_j^* , R_{\theta_j} \xi_j' ,a_j   } (y )  }\right)
$$
where the functions $z_\alpha$ are defined in \equ{capitalzeta0}, \equ{capitalzetaj}, \equ{capitalzeta2}, \equ{chico1}, \equ{chico2}, and the function $\eta$ is defined in \equ{defetaA}.
Consider furthermore their
$H_0^1(\Omega_\ve)$-projections $Z_{\alpha j}$, namely the unique
solutions of $$ \Delta Z_{\alpha j} = \Delta \bar Z_{\alpha j}
 \quad\hbox{in } \Omega_\ve, \quad Z_{\alpha j} = 0
\quad\hbox{on } \partial \Omega_\ve . $$

The nonlinear
projected problem that we first solve consists in  finding a function $\phi$ such that the following equation holds \bea \left\{
\begin{array}{ll}
 \Delta
 ( v_0 +   \phi) +
 ( v_0 +  \phi )_{+}^{p+\ve} =
\sum_{\alpha ,j} c_{\alpha j} v_j^{p-1} Z_{\alpha j } &\mbox{ in } \Omega_\ve, \\
 \phi = 0 &\mbox{ on } \partial \Omega_\ve, \\
\int_{\Omega_\ve}
 \phi  v_j^{p-1} Z_{\alpha j }  = 0 & \mbox{ for all } \alpha ,j,
 \end{array}
 \right.
 \label{P1}\eea
 for some constants
$c_{\alpha j}$,
where
$$
v_i (y) = \ve^{1\over 2} PQ_i (\ve^{1\over n-1} y ) , \quad i=1,2.
$$
A fundamental tool to solve properly Problem \equ{P1} consists in developing an invertibility theory for
the following linear problem. Given $h\in C^\alpha(\bar
\Omega_\ve)$, find a function $\phi$ such that for certain
constants $c_{\alpha j}$, $j=1,2$, $\alpha = 0  , \ldots , 3n-1 $ one has \bea
 \qquad
               \left\{ \begin{array}{ll}
 \Delta \phi + (p+\ve)v_0^{p+\ve -1} \phi = h + \sum_{\alpha ,j} c_{\alpha j} v_i^{p-1}Z_{\alpha j}
&\mbox{ in } \Omega_\ve\\
                        \phi=0     &   \mbox{ on } \partial \Omega_\ve\\
                      \int_{\Omega_\ve} v_i^{p-1} Z_{\alpha j}\phi \, dy   =0 &
                      \mbox{ for all } i, j.
                          \end{array}
                \right.
\label{P'}\eea
We do solve \equ{P'} in proper
weighted $L^\infty$-norms: for
a function $\psi$ defined on $ \Omega_\ve$, we define
\begin{eqnarray*}
 \| \psi\|_* &=&
 \sup_{x\in \Omega_\ve } | \left ( \sum_{j=1}^2 ( 1 +|x-\xi_j' |^2)^{-{n-2 \over 2}}
\right )^{- \beta} \psi
(x) | \\
& + &  \sup_{x\in \Omega_\ve } | \left ( \sum_{j=1}^2 ( 1 +|x-\xi_j' |^2)^{-{n-2 \over 2}}
\right )^{- \beta-{1\over n-2}}  D \psi
(x) | ,
\end{eqnarray*}
 where $\beta=1$ if $n=3$ and $\beta = {{2} \over n-2}$ otherwise,
and
$$ \| \psi\|_{**} =
 \sup_{x\in \Omega_\ve } | \left ( \sum_{j=1}^2  ( 1 +|x-\xi_j' |^2)^{-{n-2\over 2}}
 \right  )^{-{4\over n-2}}
\psi (x) | . $$
In Section \ref{lin}, we shall establish

\begin{proposition}\label{linprop}
Assume constraints \equ{constra1} hold on the parameter sets $A_1$ and $A_2$. Then there are numbers $\ve_0 >0$,
$C>0$, such that  for all $0<\ve < \ve_0$ and all $h\in C^\alpha
(\bar \Omega_\ve )$, problem \equ{P'} admits a unique solution
$\phi \equiv L_\ve (h)$. Besides, \be \label{bdd} \|L_\ve (h)\|_* \le C \| h
\|_{**}  , \quad  |c_{\alpha j}|\le C \| h \|_{**}
\ee
and
\be \label{bddc1}
\|\nabla_{\Lambda , \xi', a , \theta } \phi \|_{*}\le C\|h\|_{**}.
\ee
\end{proposition}

\medskip
A fixed point argument using contraction mapping Theorem gives as a direct byproduct of the previous Proposition the following result that states unique solvability of the non linear Problem \equ{P1}, for
any given sets of parameters $A_1$ and $A_2$.

\medskip

\begin{proposition} \label{nonlin}
Assume constraints \equ{constra1} hold on the parameter sets $A_1$ and $A_2$. Then there
is a constant $C>0$, such that  for all small $\ve$ there exists a unique
solution $\phi = \phi(\xi',\Lambda)$ to Problem \equ{P1} with
\begin{equation}
\label{stimanonlin} ||\phi ||_{*} \leq C \ve, \quad \| \nabla_{\Lambda , \xi', a , \theta }  \phi \|_* \leq C \ve
\end{equation}

\end{proposition}

The proof of Proposition \ref{nonlin} is postponed to Section \ref{nonlinsec}.

\medskip
Looking back at Problem \equ{P1}, it is immediate to observe that $v_0 + \phi$ is a solution to the scaled Problem \equ{eqex} if the constants $c_{\alpha j}$ appearing in \equ{P1} are all zero. The second step in our argument consists in showing that the constants $c_{\alpha j}$ can be made all equal to zero provided the parameter sets $A_1$ and $A_2$ are properly chosen. Let us explain this second part of our argument.

Consider the
function of $A_1 $ and $A_2$ defined by \be  I(A_1 , A_2 ) \equiv J_\ve ( (1+ \zeta ) \, \left( u_0 (x) + \tilde \phi (x) \right) ), \label{I}\ee
where $J_\ve$ is defined in \equ{Jeps}, $u_0$ is given by \equ{defu0} and $\tilde \phi (x) = \ve^{-{1\over 2}} \phi (\ve^{-{1\over n-2}} x) $, where $\phi$ is the unique solution to Problem \equ{P1} as predicted by Proposition \ref{nonlin}.  Recall that
$$
A_j = (\Lambda_j , \xi_j , a_j , \theta_j ) \in \R_+ \times \Omega \times \R^2 \times K , \quad j=1,2.
$$
A consequence of Proposition \ref{nonlin}, and estimate \equ{stimanonlin}, is that the function $I$
depends in a $C^1$ sense on the parameters $A_1$ and $A_2$. Furthermore,
we see that
$$
I(A_1 , A_2  ) =  (1+\zeta )^2 F_\ve ( v_0 + \phi ),
$$
where $\zeta$ is defined in \equ{defzetamm} and
$$
F_\ve (v) = {1 \over 2} \int_{\Omega_\ve } |Dv|^2 - {1 \over
{p+\ve +1}} \int_{\Omega_\ve } v^{p+\ve +1}. $$
The key observation of our argument is the following
\begin{lemma}\label{unolemma}
$u = (1+\zeta ) ( u_0+ \tilde \phi )$ is a solution of problem
\equ{P} if and only if  $(A_1 , A_2 )$ is a critical point of
$I$.
\end{lemma}

\medskip
Before giving the proof of Lemma \ref{unolemma}, an observation is in order.
Let us recall that, for $A= (\la , \xi , a , \theta ) \in \R_+ \times \R^n \times \R^2 \times \R^{2n-3}$,
the function $Q_A$ is defined as
$$
Q_A (x) = \la^{-{n-2 \over 2}} \left| {x- \hat \xi \over |x- \hat \xi |} - \hat a {|x- \hat \xi | \over \la } \right|^{2-n} \,
Q \left( {  {x-\hat \xi  \over \la } - \hat a |{x- \hat \xi \over \la }|^2 \over 1-2 \hat a \cdot ({x-\hat \xi \over \la } ) + |a|^2 |{x-\hat \xi \over \la }|^2 }\right),
$$
with
$$
\hat \xi = R_\theta \xi , \quad \hat a = R_\theta a.
$$

Recall the functions $z_\alpha $, $\alpha =0 , \ldots , 3n-1$, defined in \equ{capitalzeta0}, \equ{capitalzetaj}, \equ{capitalzeta2}, \equ{chico1}, \equ{chico2}. These are the only elements in the kernel of the linear operator
$L(\varphi ) = \Delta \varphi + p |Q|^{p-2} Q \varphi $ (see \cite{mw}).
We are in a position to prove

\begin{proof}[Proof of Lemma \ref{unolemma}] Consider, for example, the derivative of $I$ with respect to
$a_{11}$, the first component of $a_1 = (a_{11} , a_{12} ) \in \R^2$. We start observing that  that ${\partial \over \partial a_{11} } I  = 0$ is equivalent to say that ${\partial \over \partial a_{11}} F_\ve (v_0+ \phi ) = 0$.
 Next we compute ${\partial \over \partial a_{11}} F_\ve (v_0+ \phi )  = DF_\ve (v_0 + \phi ) [ {\partial \over \partial a_{11}}  v_0 + {\partial \over \partial a_{11}}  \phi ] $. On the other hand, using \equ{z211}, one sees that
$$
{\partial \over \partial a_{11}} v_0 = \bar Z_{n+2 , 1} = Z_{n+2 , 1} + o(1)
$$
where
$$ \bar Z_{\alpha j} = (\Lambda_j^*)^{-{n-2 \over 2}} \left| \eta_{\Lambda_j^* , R_{\theta_j} \xi_j' ,a_j   } (y ) \right|^{2-n} z_{\alpha} \left(  { {y-R_{\theta_j} \xi_j' \over \Lambda_j^* } - R_{\theta_j} a_j |{y- R_{\theta_j} \xi_j' \over \Lambda_j^* }|^2 \over  \eta_{\Lambda_j^* , R_{\theta_j} \xi_j' ,a_j   } (y )  }\right)
$$
and $Z_{\alpha j}$ is the $H_0^1 (\Omega_\ve )$-projection of $\bar Z_{\alpha j}$.
Taking into account that $ \|  {\partial \over \partial a_{11}}  \phi  \|_* = o(1)$, as $\ve \to 0$, we get that
${\partial \over \partial a_{11} } I  = 0$ is equivalent to say $  DF_\ve (v_0 + \phi ) [ Z_{n+2 , 1}+ o(1) ] =0.$
Based on this argument, we can say that the conditions  $\nabla I (A_1 , A_2 ) = 0$ are equivalent to say
that
\begin{equation}
\label{spero}
D F_\ve (v_0 + \phi ) [ Z_{\alpha j} +o(1) ] = 0
\end{equation}
for all $\alpha $, for all $j$.
Digging further in the above set of equalities, and using the fact that, by definition,
$DF_\ve (v_0 + \phi ) [g] = 0$ for all functions such that $\int_{\Omega_\ve } v_0^{p-1} Z_{\beta k} g = 0$,
we can be more precise in \equ{spero}: indeed, one has that
$$
D F_\ve (v_0 + \phi ) [ Z_{\alpha j} +o(1) \Theta ] = 0 \quad {\mbox {for all}} \quad \alpha , j
$$
where $\Theta $ is a uniformly bounded function, that belongs to the vector space generated by the functions $Z_{\beta , i}$. From the above relation we thus conclude that the $6n$ conditions $\nabla I (A_1 , A_2 ) = 0$ are equivalent to the $6n$ conditions
 $$D F_\ve ( v_0 +  \phi) [ Z_{\alpha j} ] = 0 $$
for all $\alpha ,j$. By definition of the $c_{\alpha j}$, it is easily seen
that this is indeed equivalent to $c_{\alpha j} = 0$ for all $\alpha ,j$. This concludes the proof of the Lemma.

\end{proof}

The result of Lemma \ref{unolemma} says that the function defined in \equ{soltutta}
$$
u(x)  = (1+ \zeta ) \, \left( u_0 (x) + \tilde \phi (x) \right),
$$
where $\tilde \phi (x)  = \ve^{-{1\over 2}} \phi (\ve^{-{1\over n-2}} x)$   and $\phi$ is the unique solution to Problem \equ{P1} as predicted by Proposition \ref{nonlin}, is a solution to \equ{P} if $(A_1 , A_2)$ is a critical point for $I$ defined in \equ{I}.

Our purpose is thus  to establish the existence of a critical point for $I(A_1 , A_2)$.
To this purpose, we first give an asymptotic estimate for the function $I(A_1 , A_2)$.
We prove
\begin{proposition}\label{expafinal}
Let $\zeta$ be given by \equ{defzetamm}. Then we have the expansion,
\begin{eqnarray}
\ve^{2\zeta -1}I(A_1 , A_2 ) &=&  2\gamma_n + \chi_n \ve\log\ve + \eta_n \ve \nonumber \\
&+&  w_n \ve \, \Psi (A_1 , A_2 )
 + o(\ve ) \theta (A_1 , A_2) , \label{efun}\end{eqnarray}
uniformly with respect to $( A_1 , A_2) $ satisfying constaint \equ{constra1}, where $\theta$
and its derivatived $D\theta$ are smooth functions that are
 uniformly bounded,
independently of $\ve$. Here, we recall
\begin{eqnarray*}  \Psi (A_1 , A_2)  &=&
 {1 \over 2}  H(\hat \xi_1, \hat \xi_1) Q^2 (\hat a_1 ) \Lambda_1^{2} +
 {1\over 2} H(\hat \xi_2,\hat  \xi_2) Q^2 (\hat a_2 ) \Lambda_2^{2}  \nonumber \\
&-&
 G(\hat \xi_1, \hat \xi_2) Q(\hat a_1 ) Q(\hat a_2 )
 \Lambda_1 \Lambda_2 +
 \log \Lambda_1\Lambda_2 ,
 \end{eqnarray*}
and the constants in \equ{efun} are those in Lemma 3.2.
\end{proposition}
The proof of this result is postponed to the end of Section \ref{nonlinsec}.

The final argument to get our Theorem \ref{teo} is to show that the function $\Psi$ in \equ{efun} defined also in \equ{psi1}
has a critical point, in fact a robust critical point of min max type, that persists under small $C^1$ perturbation.
This is where we need that our domain $\Omega$ has the shape of a smooth bounded connected domain
with a small removed hole. We show the existence of a min max structure for $\Psi$ in Section \ref{minmaxsec}. This completes the proof of Theorem \ref{teo}.

\medskip
The rest of the paper is devoted to give detailed proofs of all our previous statements.

\section{The linear problem: proof of Proposition \ref{linprop}}\label{lin}

\begin{proof}[Proof of Proposition \ref{linprop}]
The proof of this result is divided into two steps: we first assume the existence of a solution, and we prove the estimates \equ{bdd}, then we show existence of $\phi$.

To prove \equ{bdd},
 assume that  there exists
sequence $\ve = \ve_n\to 0$ such that there are functions
$\phi_\ve$ and $h_\ve$ with $\| h_\ve \|_{**} = o(1) $ such that
$$
\Delta \phi_\ve + (p+\ve) v_0^{p-1+\ve} \phi_\ve  = h_\ve
+\sum_{\alpha ,j} c_{\alpha j} v_j^{p-1} Z_{\alpha j} \quad\hbox{in } \Omega_\ve
$$
$$ \phi_\ve  = 0 \quad\hbox{on } \partial \Omega_\ve
, $$
$$
                      \int_{\Omega_\ve }v_j^{p-1} Z_{\alpha j} \phi_\ve \, dx   =0
                      \mbox{ for all } \alpha, j,
$$
for certain constants $c_{\alpha j}$, depending  on $\ve$. We shall show that  $\|\phi_\ve  \|_*  \to 0$.

We first establish that
\begin{eqnarray*}
\| \phi_\ve \|_\rho &=&
 \sup_{x\in \Omega_\ve } | \left (\sum_{j=1}^2 ( 1 +|x-\xi_j' |^2)^{-{n-2 \over 2}}
 \right )^{-(\beta-\rho)}
\phi_\ve  (x) | \\
&+&  \sup_{x\in \Omega_\ve } | \left (\sum_{j=1}^2 ( 1 +|x-\xi_j' |^2)^{-{n-2 \over 2}}
 \right )^{-(\beta-\rho - {1\over n-2})}
D \phi_\ve  (x) |  \to 0
\end{eqnarray*}
with $\rho >0$ a small fixed number. To do this, we assume the
opposite, so that with no loss of generality we may take
$\|\phi_\ve  \|_\rho  = 1$. Testing the above equation against
$Z_{\beta k}$, integrating by parts twice we get that \begin{eqnarray}
 \sum c_{\alpha j} \int_{\Omega_\ve}
v_j^{p-1}Z_{\alpha j}  Z_{\beta k} & = & \int_{\Omega_\ve } [ \Delta  Z_{\beta k} + (p+\ve) v_0^{p-1+\ve
} Z_{\beta k} \phi ] \nonumber \\
& -&  \int_{\Omega_\ve } h_\ve Z_{\beta , k} . \label{c1} \end{eqnarray}
Formula \equ{c1} defines a linear system in the $6n$ variables $c_{\alpha , j}$, which is
uniquely solvable, with bounded inverse. This is due to the following facts: fix $\alpha $ and $j$.
If $k \not= j$, then
$$
 \int_{\Omega_\ve}
v_j^{p-1}Z_{\alpha j}  Z_{\beta k} = O(\ve^{n \over n-2} ) , \quad {\mbox {for any}} \quad \beta.
$$
If $k=j$, then we have
$$
 \int_{\Omega_\ve}
v_j^{p-1}Z_{\alpha j}  Z_{\beta j} = \left\{ \begin{matrix}   \int_{\R^n} |Q|^{p-1} z_\alpha^2 + O(\ve^{n\over n-2} )   & \hbox{ if } \alpha = \beta \, , \\ & \\   \int_{\R^n} |Q|^{p-1} z_1 z_{n+2}  + O(\ve^{n\over n-2} )    & \hbox{ if } \alpha = 1, \beta= n+2
\, , \\ & \\      \int_{\R^n} |Q|^{p-1} z_2 z_{n+3}  + O(\ve^{n\over n-2} )  & \hbox{ if } \alpha = 2, \beta = n+3
\, , \\ & \\     O(\ve^{n\over n-2} )   & \hbox{otherwise.}
 \end{matrix}
  \right. \
$$
Moreover the numbers $\int_{\R^n} |Q|^{p-1} z_\alpha^2 $, $\alpha = 0 , 1 , \ldots , 3n-1$, and
$ \int_{\R^n} |Q|^{p-1} z_1 z_{n+2} $, $ \int_{\R^n} |Q|^{p-1} z_2 z_{n+3} $ are fixed numbers, different from zero, that are independent of $\ve$.
The above computations tell us several facts: first of all, the linear system \equ{c1} of $6n$ equations
in the $6n$ variables $c_{\alpha j}$ at main order decouples in two systems, the first a $3n \times 3n$ system in the variables $c_{\alpha ,1}$ and the second in a $3n \times 3n$ system in the variables $c_{\alpha 2}$.
Second, if we analyze for instant the system in the $c_{\alpha 1}$, we see that the coefficients $c_{01} , \ldots c_{(3n-1) 1}$ are coupled, at main order, but the coupling is very clear, in fact only two coupling occurs: the variable $c_{11}$ with the variable $c_{n+2 , 1 }$ and the variable $c_{21}$ with the variable $c_{n+3 , 1}$. Except for these coupling, the system in the $c_{\alpha 1}$ decouples at main order, as $\ve \to 0$.
We finally observe that the matrices
\begin{eqnarray}
\label{decoupling}
& & \left[ \begin{matrix}   \int_{\R^n} |Q|^{p-1} z_1^2   &  \int_{\R^n} |Q|^{p-1} z_1 z_{n+2} \\
 & \\
 \int_{\R^n} |Q|^{p-1} z_1 z_{n+2} &  \int_{\R^n} |Q|^{p-1} z^2_{n+2} \end{matrix} \right], \quad  \nonumber \\
& & \\
& & \left[ \begin{matrix}   \int_{\R^n} |Q|^{p-1} z_2^2   &  \int_{\R^n} |Q|^{p-1} z_2 z_{n+3} \\
& \\
 \int_{\R^n} |Q|^{p-1} z_2 z_{n+3} &  \int_{\R^n} |Q|^{p-1} z^2_{n+3} \end{matrix} \right] \nonumber
\end{eqnarray}
are invertible. Thus we conclude that \equ{c1} defines a linear system in the $6n$ variables $c_{\alpha , j}$, which is
uniquely solvable, with bounded inverse.

On the other hand, it is easy to see that
we have, for $l=1,2$,
$$ \int_{\Omega_\ve } \Delta  Z_{\alpha k} + (p+\ve )v_0^{p+\ve -1}
Z_{\alpha k}  \phi  = o(1) \|\phi\|_{\rho} , \quad {\mbox {and}} \quad
| \int_{\Omega_\ve } h_\ve , Z_{\beta k } | \leq C \| h_\ve \|_{**} .
$$
Thus, we conclude that \be |c_{\alpha j} | \leq C \| h_\ve \|_{**}
+o(1)\|\phi_\ve\|_\rho \label{cij}\ee so that $c_{ij} = o(1)$.
Let  $G_\ve$ denotes the Green's function of $\Omega_\ve$. We have for $x \in \Omega_\ve $
\begin{eqnarray}
\phi_\ve (x)&=&  (p+\ve)\int_{\Omega_\ve} G_\ve (x,y)  v_0^{p+\ve -1}
\phi_\ve dy   \nonumber \\
& -&  \int_{\Omega_\ve} G_\ve (x,y) h_\ve \, dy - \sum c_{\alpha j} \int_{\Omega_\ve}
 v_j^{p-1}
Z_{\alpha j} G_\ve (x,y)\, dy  \label{cc4} \end{eqnarray}
Furthermore, the function $\phi_\ve$ is of class
$C^1$ and
$$
\partial_{x_j} \phi_\ve (x)= p \int_{\Omega_\ve} \partial_{x_j} G_\ve (x,y)  v_0^{p+\ve -1}
\phi_\ve dy   = $$
\begin{equation}
 - \int_{\Omega_\ve } \partial_{x_j } G_\ve (x,y) h_\ve \, dy - \sum c_{\alpha j} \int_{\Omega_\ve}
 v_0^{p-1}
Z_{\alpha j} \partial_{x_j} G_\ve (x,y)\, dy \quad x\in \Omega_\ve .
\label{derivata}
\end{equation}
Direct estimates give
\begin{eqnarray*}
 \int_{\Omega_\ve} G_\ve (x,y) |h_\ve|\, dy   &\le &
 \|h_\ve\|_{**} C \int_{\R^n } \Gamma (x- y)
 \sum_{j=1}^2 {1\over  ( 1 +|y-\xi_j' |^2)^{2}}
 \, dy  \\
&  \le &
C \|h_\ve\|_{**}   \sum_{j=1}^2 \left( {1\over  ( 1 +|x-\xi_1' |^2)^{{n-2\over
2}}} \right )^\beta ,
\end{eqnarray*}
\begin{eqnarray*}
 | \int_{\Omega_\ve}
v_j^{p-1} Z_{\alpha j} G_\ve (x,y)\, dy|& \le &C
   \int_{\R^n } \Gamma (x- y)
 \sum_{j=1}^2 {1 \over  ( 1 +|y-\xi_i' |^2)^{{{n+3} \over 2}}} \\
& \le &
C  \sum_{j=1}^2 {1\over  ( 1 +|x-\xi_1'
|^2)^{{n-2\over 2}} }
\end{eqnarray*}
and $$
 \int_{\Omega_\ve}
G_\ve (x,y) v_0^{p+\ve-1} |\phi_\ve| dy    \le  C
\|\phi_\ve\|_{\rho}  \sum_{j=1}^2 \left( {1\over  ( 1 +|x-\xi_j' |^2)^{{n-2\over 2}} } \right)^\beta  . $$
Analogously we get
\begin{eqnarray*}
 \int_{\Omega_\ve} |\partial_{x_j } G_\ve (x,y) h|\, dy
&\le & \|h \|_{**} C \sum_j \int_{\R^n  } {1\over |x- y|^{n-1}}
  ( 1 +|y-\xi_j' |^2)^{-2} \, dy\\
& \le & C \|h \|_{**}
\sum_{j=1}^2 \left( {1\over  ( 1 +|x-\xi_1' |^2)^{{n-2\over
2}}} \right )^{\beta + {1\over n-2} }  ,
\end{eqnarray*}
\begin{eqnarray*}
 | \int_{\Omega_\ve}
v_0^{p-1} Z_{\alpha  j} \partial_{x_j} G_\ve (x,y)\, dy|
&\le &
C( \|\phi_\ve \|_\rho + \|h \|_{**})
 \sum  \int_{\R^n  } {1\over |x- y|^{n-1}}
 \left ( ( 1 +|y-\xi_i' |^2)^{-{{n+3} \over 2}} \right )
 \\
& \le &
C( \|\phi_\ve \|_\rho + \|h_\ve \|_{**})
\sum_{j=1}^2 \left( {1\over  ( 1 +|x-\xi_1' |^2)^{{n-2\over
2}}} \right )^{\beta +{1\over n-2}}
\end{eqnarray*}
and
\begin{eqnarray*}
 \int_{\Omega_\ve}
|\partial_{x_j} G_\ve (x,y) v_0^{p+\ve-1} \phi_\ve | dy     & \le & C
\|\phi_\ve \|_{\rho}  \sum_{j=1}^2 \left( {1\over  ( 1 +|x-\xi_1' |^2)^{{n-2\over
2}}} \right )^{\beta + {1\over n-2}} .
\end{eqnarray*}
Equation
\equ{cc4} and the above estimates imply that
$$ |\phi_\ve (x)|
\le  C( \|\phi_\ve\|_\rho + \|h_\ve\|_{**})
 \left (\sum_{j=1}^2 {1 \over ( 1
+|x-\xi_j' |^2)^{{n-2\over 2}} } \right)^\beta
$$
and
$$ |D \phi_\ve (x)|
\le  C( \|\phi_\ve\|_\rho + \|h_\ve\|_{**})
 \left (\sum_{j=1}^2 {1 \over ( 1
+|x-\xi_j' |^2)^{{n-2\over 2}} } \right)^{\beta + {1\over n-2}} .
$$
In particular
$$
\left (\sum_{j=1}^2 {1\over  ( 1 +|x-\xi_j' |^2)^{{n-2\over 2}}}  \right )^{-(\beta-\rho)} |\phi_\ve (x)| \le
 C
\left ( \sum_{j=1}^2 {1\over  ( 1 +|x-\xi_j' |^2)^{{n-2\over 2}}} \right )^\rho  .
$$
Since $\|\phi_\ve  \|_\rho  = 1$, we assume that
 $\|\phi_\ve  \|_{L^\infty (B_R(\xi_1' ))} >\gamma $ for  certain $R>0$ and $\gamma >0$ independent of $\ve$. for
either $i=1$ or $i=2$. Then local
elliptic estimates and the bounds above  yield that, up to a
subsequence, $\tilde \phi_\ve (x) = \phi_\ve (x -\xi_1' )$
converges uniformly over compacts of $\R^N$ to a nontrivial
solution $\tilde \phi$ of \be
 \Delta \tilde
\phi + p|Q|^{p-1} \tilde \phi  = 0 ,
\label{limeq1} \ee  which besides satisfies
\be |\tilde \phi (x)| \le C|x|^{(2-n)\beta} . \label{limeq2} \ee
In dimension $n=3$ this means
$
|\tilde \phi (x) | \leq C |x|^{2-n} .
$
In higher dimension, a bootstrap argument of $\tilde \phi $
solution of \equ{limeq1}, using estimate
\equ{limeq2}, gives
$
|\tilde \phi (x) | \leq C |x|^{2-n} .
$
Thanks to non degenerate result in \cite{mw},  this implies that $\tilde \phi$ is a linear
combination of the functions $ z_\alpha$, defined in \equ{capitalzeta0}, \equ{capitalzetaj}, \equ{capitalzeta2}, \equ{chico1} and \equ{chico2}. On the other
hand, dominated convergence Theorem gives that the orthogonality conditions
$ \int_{\Omega_\ve}
 \phi_\ve v_j^{p-1}  Z_{\alpha j }  = 0
$
pass to the limit, thus getting
$$
 \int_{\R^n} |Q|^{p-1} z_\alpha \tilde \phi
= 0 \quad {\mbox {for all}} \quad \alpha =0, \ldots , 3n-1. $$  Hence the only
possibility is that $\tilde \phi \equiv 0$, which is a
contradiction which yields the proof of $\|\phi_\ve\|_\rho \to 0$.
Moreover,  we observe that
$$
\|\phi_\ve \|_{*} \le C(\|h_\ve \|_{**} + \|\phi_\ve \|_{\rho} ),
$$
hence $ \|\phi_\ve \|_{*} \to 0 . $

\bigskip
Now we are in a position to prove the existence of $\phi$ solution to \equ{P'}. To do this, let
us consider the space
$$ H= \{ \phi \in H_0^1(\Omega_\ve) \ | \
\int_{\Omega_\ve } v_j^{p-1}  Z_{\alpha j } ,\phi> = 0 \ \forall \, \alpha ,j\ \}  $$ endowed
with the usual inner product $ [\phi , \psi ] = \int_{\Omega_\ve}
\nabla\phi\nabla\psi . $ Problem \equ{P'} expressed in weak form
is equivalent to that of finding a $\phi \in H$ such that $$
 [\phi , \psi ] = \int_{\Omega_\ve}  \bigl( (p+\ve) v_0^{p+\ve -1}\phi - h
  \bigl)\, \, \psi\,
 \qquad \forall \psi \quad \in H.
$$
With the aid of Riesz's representation theorem, this equation gets
rewritten in $H$ in the operational form \be \phi  = T_\ve (\phi)
+ \tilde h \label{T} \ee with certain $ \tilde h \in H$
 which depends linearly in $h$ and where $T_\ve $ is a compact
 operator in $H$.
 Fredholm's alternative guarantees unique solvability of this
 problem for any $h$ provided that the homogeneous
 equation
$ \phi = T_\ve (\phi) $ has only the zero solution in $H$.  Assume it has a nontrivial
solution $\phi =\phi_\ve$, which with no loss of generality may be
taken so that $\|\phi_\ve \|_* =1$. But for what we proved before, necessarily $\|\phi_\ve \|_*\to 0$. This
is certainly a contradiction that proves that this equation only
has the trivial solution in $H$. We conclude then that for each
$h$, problem \equ{P'} admits a unique solution.
Standard arguments give then the validity of \equ{bdd}.

\bigskip
We go now to the issue of the dependence of the solution $\phi$ to \equ{P'} on the parameters
$A_1' = (\Lambda_1 , \xi'_1 , a_1 , \theta_1 )$  and $A_2' = (\Lambda_2 , \xi'_2 , a_2 , \theta_2 )$.
Let us fix $j=1$ and define $A_1'= (A_{11} , A_{12} , \ldots , A_{1 3n} )$ the components of the vector
$A_1'$. Let us differential $\phi$ with respect to $A_{1l}$, for some $l=1, \ldots , 3n$.  We set formally
$Z= {\partial \over \partial A_{1l}} \phi$.

We define the number $b_{\alpha j} $ so that
$$
\int_{\Omega_\ve } v_i^{p-1} Z_{\beta i} [ Z - \sum_{\alpha j} b_{\alpha j }  Z_{\alpha j} ] = 0 , \quad {\mbox {for all}} \quad \beta , i .
$$
This amounts to solving a linear system in the constants $b_{\alpha j } $,
\be \sum_{\beta j} b_{\beta j } \int_{\Omega_\ve } v_i^{p-1} Z_{\alpha i }   Z_{\beta j}
 = \int_{\Omega_\ve }  {\partial \over \partial A_{1l}}  (  v_i^{p-1} Z_{\alpha i} ) \phi,
\label{ort} \ee
as a direct differentiation with respect to $A_{1l}$ of the orthogonal conditions $\int_{\Omega_\ve }
v_i^{p-1} Z_{\alpha i} \phi = 0$ directly shows. Arguing as in \equ{c1}, we see that \equ{ort} is uniquely
solvable and that
$$
b_{\beta i }  = O( \| \phi\|_{*} )
$$
uniformly for parameters $A_1'$ and $A_2'$  in the considered region.
Thus   $\eta \in H_0^1(\Omega_\ve)$ and
\be  \int_{\Omega_\ve } v_i^{p-1} Z_{\alpha i } \eta = 0\quad\hbox{ for all }\alpha ,i.
 \label{ortog}\ee
On the other hand, a direct but long computation shows that
\be
\Delta \eta    + (p+\ve ) v_0^{p-1+\ve} \eta   =
 f + \sum_{\alpha ,j}d_{\alpha j}v_j^{p-1}  Z_{\alpha j} \quad\hbox{in }
\Omega_\ve , \label{et}\ee where $d_{\alpha j} = {\partial \over \partial A_{1l}} c_{\alpha j}$ and
$$ f= \sum_{\alpha ,j} b_{\alpha j} (
-(\Delta + (p+\ve ) v_0^{p-1+\ve}) Z_{\alpha j} + c_{\alpha j}\partial_{A_{1l}}
(v_j^{p-1} Z_{\alpha j }) -
$$
\be (p+\ve ) \partial_{A_{1l}} (v_0^{p-1+\ve }  \phi ) , \label{f}\ee
Thus  we have that
$ \eta = L_\ve (f) $. Moreover,
 we
easily see that
$$\|\phi\partial_{A_{1l}} ( v_0^{p-1+\ve })\|_{**}
\le C \|\phi \|_{*} .
$$
On
the other hand
$$
|\partial_{A_{1l} } (v_i^{p-1 }Z_{\alpha i  }(x))| \le C|x-\xi_i' |^{-n-4} ,
$$
hence
$$
\| c_{\alpha i}\partial_{A_{1l}} v_i^{p-1} Z_{\alpha i } \|_{**} \le C \|h\|_{**}
$$
since we have that $c_{\alpha i } = O(\| h \|_{**} ) $.
We conclude that
$$
\|f\|_{**} \le C \|h\|_{**} .
$$
Reciprocally, if we  define
$$
Z= L_\ve (f)  + \sum_{\alpha ,j} b_{\alpha j}  v_j^{p-1} Z_{\alpha j}, $$
with $b_{\alpha j}$ given
by relations \equ{ort} and $f$ by \equ{f}, we check that indeed $Z=\partial_{A_{1l}}\phi$. In fact $Z$
depends continuously on the parameters $A_1'$, $A_2'$ and  $h$
for the norm $\|\ \|_{*}$, and $\|Z\|_{*} \le C\|h\|_{**}$ for parameters
in the considered region. The corresponding result for
differentiation with respect to the $A_2'$ follow
similarly.

In other words, we proved that  $(A_1' , A_2') \mapsto L_\ve$ is of class
$C^1$ in ${\mathcal L} ( L^{\infty}_{**} , L^{\infty}_{*})$ and, for
instance,  \be (D_{A_{1l}} L_\ve) (h) = L_\ve (f)  + \sum_{\alpha ,j} b_{\alpha j}
Z_{\alpha j}, \label{dl}\ee where $f$ is given by \equ{f} and $b_{\alpha j}$
by \equ{ort} .
This concludes the proof.

\end{proof}

\section{The non-linear Problem: proof of Proposition \ref{nonlin}} \label{nonlinsec}

\begin{proof}[Proof of Proposition \ref{nonlin}]
We write the equation in \equ{P1}
as
$$
 \Delta \phi  + (p
+\ve) v_0^{p+\ve -1} \phi = E - N_{\ve } ( \phi ) + \sum_{\alpha ,j}
c_{\alpha j} v_j^{p-1}  Z_{\alpha j } \quad \mbox{ in } \quad \Omega_\ve
$$
where \be N_\ve (\phi) = ( v_0 + \phi
)_{+}^{p+\ve}  - v_0^{p+\ve}  -
 (p +\ve)v_0^{p+\ve -1} \phi,\quad
E = v_0^{p+\ve}  - Q_1^{p} - Q_2^{p} . \label{ne}\ee
Observe that
$$
|E  | \leq C \left( |v_0^{p+\ve} - v_0^p | + |v_0^p - Q_1^p - Q_2^p| \right)
$$
$$
\leq C
\ve \left(  |Q_i|^{p}|\log  |Q_i|| (x) + {1 \over 1+ |x-\xi_i'|^4 } \right)
$$
in the regions where $|x-\xi_i' |\leq \bar \delta \ve^{ -{1 \over
N-2}}$, for small $\bar\delta>0$. Taking into account that $|E
| \leq C \ve^{n+2 \over n-2}$ in the complement of these two
regions, we get
$$
\| E \|_{**} \leq C \ve .
$$
To estimate $N_\ve (\phi)$, it is convenient, and
sufficient for our purposes, to assume $\|\phi  \|_* <1$. Note
that, if $n\leq 6$, then
$p\ge 2$ and we can estimate
$$
|N_\ve (\phi ) | \leq C |v_1 + v_2 |^{p-2} |\phi|^2
$$
and hence
$$
\| N_\ve (\phi      ) \|_{**} \leq
                        C \| \phi      \|_{*}^2 .
$$
Assume now that $n>6$. If $|\phi |\geq {1\over 2} $ we have
$$
|N_\ve (\phi ) | \leq C  |\phi|^p
$$
$$
\| N(\phi )\|_{**}
\leq   C \ve^{ -{n-6 \over 2}}\| \phi
\|_{*}^p .
$$
Let us consider now the case $|\phi | \leq {1\over 2} v_0$.
In the region where $dist (y, \partial \Omega_\ve ) \geq \delta
\ve^{-{1 \over n-2}}$, for some $\delta>0$, then $v_0 (y) \geq
\alpha_\delta U (y)$ for some $\alpha_\delta >0 $; hence in this
region, we have
$$
| N (\phi      ) | \leq C U^{2\beta -1 } \|
\phi \|_{*}^2 \leq C \ve^{(2\beta -1)}
 \|\phi     \|_{*}^2.
$$
On the other hand, when $dist (y, \partial \Omega_\ve ) \leq \delta
\ve^{-{1\over n- 2}}$, the following facts occur: $U (y), \, v_0 (y) =
O(\ve)$ and, as $y \to \partial \Omega_\ve$, $v_0 (y)
= C \ve^{n-1 \over n-2} dist (y, \partial \Omega_\ve ) (1+ o( 1 ))$.
This second assertion is a consequence of the fact that
the Green function of the domain $\Omega$ vanishes linearly with
respect to $dist (x,
\partial \Omega )$ as $x \to \partial \Omega$.
These two facts imply that, if $dist (y, \partial \Omega_\ve ) \leq
\delta \ve^{-{1\over n-2}}$ and $\phi (y) \not=0$ (otherwise $N
(\phi ) (y)=0$), then
$$
\| N(\phi ) \|_{**}  \leq  U^{-{4\over n-2}} v_0^{
p-2} | \phi |^2
$$
$$\leq
C  U^{-{4\over n-2}} \left( \ve^{n-1 \over n-2} dist (y,
\partial \Omega_\ve ) \right)^{p-2} dist (y, \partial \Omega_\ve )^2
| D \phi (\bar y) |^2
$$
$$
\leq C U^{-{ 4\over n-2}+ 2 \beta + {2\over n-2}} +\ve^{ {n-1 \over
n- 2} (p-2) -{p\over 2} } \| \phi \|_*^2 \leq C \ve^{-{n-6 \over n-2}}
\| \phi \|_*^2 .
$$
Combining these relations we get \begin{eqnarray} \| N (\phi )
\|_{**} \leq \left\{ \begin{array}{ll}
                        C \| \phi     \|_{*}^2  &\mbox{ if } n \le 6\\
          C \ve^{-{n-6 \over n-2}} \| \phi      \|_{*}^2  &\mbox{ if } n> 6.
                          \end{array}
                \right.
                \label{Ne}\end{eqnarray}

Now, we are in position to prove that problem (\ref{P1}) has a
unique solution $\phi=\widetilde\phi+\widetilde\psi$, with
\begin{equation}\label{tor}
\widetilde\psi:=  -T_\ve  (E ),
\end{equation} with the required properties.
Here $T_\ve$ denotes the linear operator defined by Proposition \ref{linprop}, namely
$T_\ve (h) = \phi $ is $L_\ve \phi = h$.
We see that problem \equ{P1} is equivalent to solving a fixed point
problem. Indeed $\phi = \tilde \phi +\tilde \psi$ is a solution of
\equ{P1} if and only if
$$
\tilde \phi = - T_\ve (N (\tilde \phi + \tilde \psi )) \equiv A_\ve (\tilde
\phi ).
$$
We proceed to prove that the operator $A_\ve $ defined above is a
contraction inside a properly chosen region. Since $\| E \|_* \leq C \ve$, the result of Proposition \ref{linprop} gives that
$$
\| \tilde \psi \|_{**} \leq C \ve
$$
and
\begin{eqnarray}
\|N (\tilde \psi + \eta)\|_{**} \le \left\{ \begin{array}{ll}
                        C( \ve^2 + \ve \| \eta \|_* + \| \eta \|_*^2 ) &\mbox{ if } n\le  6\\
          C(  \ve^{1+{4\over n-2}} + \ve^{4\over n-2} \| \eta \|_* + \ve^{-{n-6 \over n-2}}\| \eta \|_*^2  ) &\mbox{ if } n> 6.
                          \end{array}
                \right.
                \label{Ne1}\end{eqnarray}
Call
$$F = {{\{ \eta \in H_0^1 \, : \, ||\eta ||_{*} \leq
R \ve  \}}}.
  $$
>From Proposition \ref{linprop} and \equ{Ne1} we conclude that, for $\ve$
sufficiently small and any $\eta \in l F$ we have $$ \| A_\ve
(\eta ) \|_{*}  \le  C \ve . $$ If we choose $R$ big enough in the definition of $F$, we get then that $A_\ve$ maps $F$ in itself.
Now we will show that the map $A_\ve $ is a contraction, for any $\ve$
small enough. That will imply that $A_\ve $ has a unique fixed point in
$F$ and hence problem \equ{P1} has a unique solution.
For any $\eta_1 $, $ \eta_2 $ in $F$ we have
$$ \| A_\ve (\eta_1 ) - A_\ve (\eta_2 ) \|_{*} \leq C  \| N_\ve
(\tilde \psi  + \eta_1 ) - N_\ve (\tilde \psi  + \eta_2 ) \|_{**} , $$
hence we just need to check that $N$ is a contraction in its
corresponding norms. By definition of $N$
$$
D_{\bar\eta} N_\ve  (\bar\eta) = (p+\ve) [ ( v_0+ \bar \eta )_{+}^{p+\ve-1}
- v_0^{p+\ve -1} ].
$$
Hence we get
$$ |N_\ve  (\tilde \psi  + \eta_1 ) - N_\ve (\tilde \psi  +\eta_2 ) | \le
C\bar v_0^{p-2} |\bar \eta | |\eta_1 - \eta_2| .
$$
for some $\bar\eta$ in the segment joining $\tilde \psi  + \eta_1$
and $\tilde \psi + \eta_2$. Hence, we get for small enough $\|\bar
\eta\|_{*}$,
$$ \|N  (\tilde \psi  +\eta_1 ) - N (
\tilde \psi  +\eta_2 ) \|_{**} \le C\ve^{p-2 + 2\beta} \| \bar \eta
\|_{*} \| \eta_1 - \eta_2\|_{*} .
$$
We conclude that there exists $c\in (0,1)$ such that
$$
\|N  (\tilde \psi  +\eta_1 ) - N ( \tilde \psi  + \eta_2 ) \|_{**}
\le c \|\eta_1 - \eta_2\|_{*}.
$$
This concludes the proof of existence of $\phi$ solution to \equ{P1}, and the first estimate in \equ{stimanonlin}. We devote the rest to prove the second estimate in \equ{stimanonlin}.

We recall that $\phi$ is defined through the
relation
 $$B(A_1 , A_2 ,\phi) \equiv
 \phi + L_{\ve } ( N_\ve (\phi+\psi  ) ) =0. $$
We have that
 $$D_\phi B(A_1 , A_2 ,\phi) [\theta ] =
 \theta  + L_{\ve } (\theta D_{\bar\phi}N (\phi+\psi  ) ) \equiv
 \theta + M(\theta) $$
where
$$
D_{\bar \phi } N(A_1 , A_2 , \bar\phi) = (p+\ve) [ ( v_0+ \bar
\phi )_{+}^{p+\ve-1}  - v_0^{p+\ve -1} ].
$$
Now,
 $$
 \| M(\theta) \|_{*}\le
 C\|(\theta D_{\bar\phi}N (\phi+\psi  ) )\|_{**}\le
  C\|v_0^{-{4\over n-2} +\beta } D_{\bar\phi}N (\phi+\psi  ) )\|_{\infty}
  \|\theta\|_{*},
  $$
and
  $$
 \bar v_0^{-{4\over N-2} +\beta }
  |D_{\bar\phi}N_\ve (\phi+\psi  ) )| \le
 v_0^{2\beta -1 } \| \phi +\psi \|_{*}
\le
 C
  \ve^{\min\{2\beta , 1\}}.
$$
It follows that for small $\ve$, the linear operator
 $D_\phi B(A_1 , A_2  ,\phi) $ is invertible in $L^\infty_{*}$,
 with uniformly bounded inverse. It also depends continuously
 on its parameters.
Define again
$A_1' = (\Lambda_1 , \xi'_1 , a_1 , \theta_1 )$  and $A_2' = (\Lambda_2 , \xi'_2 , a_2 , \theta_2 )$.
Let us fix $j=1$ and define $A_1'= (A_{11} , A_{12} , \ldots , A_{1 3n} )$ the components of the vector
$A_1'$. Let us differential $\phi$ with respect to $A_{1l}$, for some $l=1, \ldots , 3n$.
We have
 \begin{eqnarray} D_{A_{1l}} N(A_1 , A_2 , \bar\phi) &=& (p+\ve) [( v_0+ \bar
\phi )_{+}^{p+\ve-1}  - v_0^{p+\ve-1}
\nonumber \\
&-& (p +\ve-1)v_0^{p+\ve -2} \bar
\phi ]D_{A_{1l} } v_0  . \label{for}\end{eqnarray}
and
 $$D_{A_{1l}}  B(A_1 , A_2 ,\phi) =
 (D_{A_{1l}} L_\ve) ( N (\phi+\psi  )) +
 $$
 $$
 \bigl[ L_\ve( (D_{A_{1l}} N)(A_1 , A_2 , \phi + \psi  )) +
  L_\ve( (D_{\bar\phi}N)(A_1 , A_2 , \phi + \psi  ) D_{A_{1l}}\psi
  ) \bigl].
$$
 Here
 $ D_{A_{1l}} L_\ve$ is the operator defined by the expression  \equ{dl}
and the second quantity by \equ{for}. Observe also that \be
 D_{A_{1l} }\psi   = (D_{A_{1l} } L_\ve) (E ) +  L_\ve ( D_{A_{1l} } E).
 \label{dpsi}\ee
see \equ{tor}.
 Also,
 \be D_{A_{1l}} E = (p+\ve)
 v_0^{p+\ve -1}
 D_{A_{1l}}v_1 -p v_1^{p -1}  D_{A_{1l} } v_1.
\label{dr}\ee These expressions  also depend continuously on their
parameters.

The implicit function theorem then applies to yield that
$\phi(A_1 , A_2 )$ indeed defines a $C^1$ function into
$L^\infty_{*}$. Moreover, we have for instance
 $$D_{A_{1l} }\phi =
 -(D_\phi B(A_1 , A_2 ,\phi))^{-1} \bigl[
 (D_{A_{1l}} L_\ve) ( N (\phi+\psi  )) +
[ L_\ve( D_{A_{1l} } [N(A_1 , A_2 , \phi + \psi  )])\,  +
 $$
 $$
 L_\ve( (D_{\bar\phi}N)(A_1 , A_2  , \phi + \psi  ) D_{A_{1l} }\psi
 )]  \bigl].
 $$
 Hence,
 $$\| D_{A_{1l} }\phi\|_{*}\le C (
 \| N (\phi+\psi  )\|_{**} +
 $$
 $$
 \| D_{A_{1l} } N(A_1 , A_2  , \phi + \psi  ) \|_{**}
 +\| D_{\bar\phi} N(A_1 , A_2 , \psi  + \phi) D_{A_{1l} }\psi  \|_{**} ),
$$
thanks to \equ{dl}. On the other hand, we get
 \be
 \| N_\ve (\phi+\psi  )\|_{**}
 \le
\left\{ \begin{array}{ll}
C\ve^2   &\mbox{ if } n\leq 6 \\
C \ve^{p\beta +1 } &\mbox{ if }  n>6.
 \end{array}
 \right.
\label{rel}\ee Thus, from \equ{for} we have
$$
|(D_{A_{1l} } N)(A_1 , A_2  , \bar \phi)|\le C\bar v_0^{n-1\over
n-2} |(v_0+ \bar \phi )_{+}^{p+\ve-1}  - v_0^{p+\ve-1} - (p
+\ve-1)v_0^{p+\ve -2} \bar \phi | \le
$$
$$
C\bar v_0^{{5\over n-2} +\ve +\beta } \| \bar \phi \|_{*} ,
$$
hence
$$
\|(D_{A_{1l} } N)(A_1 , A_2 ,\psi  +\phi)\|_{**} \le
 C\|\phi +\psi  \|_{*}
\le C\ve.
$$
In similar way we get that
 $$
 \| D_{\bar\phi} N(A_1 , A_2 , \psi+ \phi) D_{A_{1l}} \psi  \|_{**}
\le C\ve.
 $$
Hence, we finally get
$$
 \| D_{A_{1l}} \phi\|_{*}\le C \ve ,
$$
as desired. A similar estimate holds for differentiation with
respect to the other variables. This concludes the proof.

\end{proof}

\bigskip
\begin{proof}[Proof of Proposition \ref{expafinal}]

We write
\begin{eqnarray*}
I(A_1 , A_2 )&=& J_\ve ((1+\zeta ) (u_0 + \tilde \phi )) - J_\ve ( (u_0 + \tilde \phi ))\\
&+& J_\ve ( (u_0 + \tilde \phi )) - J_\ve ( PQ_1 + PQ_2).
\end{eqnarray*}
Since  $\tilde \phi (x) = \ve^{-{1\over 2}} \phi (\ve^{-{1\over n-2}} x) $, and $\| \phi \|_* \leq C\ve$, we have that
\begin{eqnarray*}
J_\ve ((1+\zeta ) (u_0 + \tilde \phi )) - J_\ve ( (u_0 + \tilde \phi )) &=&
J_\ve ((1+\zeta ) (u_0  )) - J_\ve ( (u_0  )) + o(\ve )
\end{eqnarray*}
At this point, arguing like in the proof of Lemma \ref{paris} we are able to show that
$$
J_\ve ((1+\zeta ) (u_0  )) - J_\ve ( (u_0  ))  = 2 \tilde \gamma_n \ve \log \ve + O(\ve^2 |\log \ve |),
$$
where $\tilde \gamma_n$ is a fixed constant, independent of $\ve$. Observe also that
$$
\nabla_{A_1 , A_2} \left[ J_\ve ((1+\zeta ) (u_0  )) -  J_\ve ( (u_0  ))  \right]= O(\ve^2 |\log \ve |)
$$
uniformly for parameters $A_1$ and $A_2$ in the considered region.

Given the result of Lemma \ref{paris}, we need to show that
  \be I(A_1 , A_2 ) - J_\ve ( PQ_1 + PQ_2) = o(\ve)
\label{ee5}\ee and \be \nabla_{A_1 , A_2 }[I(A_1 , A_2 ) -
J_\ve ( PQ_1 + PQ_2)] = o(\ve). \label{e5} \ee

Recall now that $u_0 = PQ_1 + PQ_2$. Let us  apply a  Taylor expansion
\be
J_\ve ( (u_0 + \tilde \phi )) - J_\ve ( PQ_1 + PQ_2) = \int_0^1 t dt D^2J_\ve(u_0 +t\tilde \phi ) [
\tilde \phi ,\tilde \phi], \label{e1} \ee since $ 0= D  F_\ve(v_0+\phi) [ \phi ] = (1+\zeta )^2
   DJ_\ve(u_0 \tilde \phi)[\tilde  \phi ] .
$ Now, from the definition of $\phi$, we see that
$$
\int_0^1 t dt D^2J_\ve(u_0  + t\tilde \phi ) [ \tilde
\phi ,\tilde \phi] = (1+\zeta )^2 \int_0^1 t dt D^2 F_\ve(v_0+ t \phi ) [ \phi , \phi] =
 $$
 $$
(1+\zeta )^2
 \int_0^1 t dt
[ \int_{\Omega_\ve} |\nabla \phi|^2  - (p+\ve ) ( v_0 + t\phi
)^{p+\ve -1} \phi^2 ] =
$$
\be (1+\zeta )^2
 \int_0^1 t dt( \int_{\Omega_\ve}  N (\phi +\psi  )
 \phi + \int_{\Omega_\ve} (p+\ve )[v_0^{p+\ve -1}-
 ( v_0+ \psi  + t\phi )^{p+\ve -1}
] \phi^2) . \label{e2} \ee Since, we recall $\|\phi\|_*
+\|\psi\|_* = O(\ve)$, the above relation together with \equ{rel}
yield in particular, \bea
 I(A_1 , A_2 ) - J_\ve ( u_0 + \tilde \phi ) =
\left\{ \begin{array}{ll}
                         O(\ve^2) &\mbox{ if } n< 6 \\
                         O(\ve^2 |\log \ve |) &\mbox{ if } n=6 \\
                        O(\ve^{1+{4 \over N-2}} ) &\mbox{ if } n\geq 7,
                          \end{array}
                \right.
\label{ee41}\eea uniformly on $A_1 , A_2 $ in the considered
region. Let us estimate now difference in derivatives.
Define again
$A_1' = (\Lambda_1 , \xi'_1 , a_1 , \theta_1 )$  and $A_2' = (\Lambda_2 , \xi'_2 , a_2 , \theta_2 )$.
Let us fix $j=1$ and define $A_1'= (A_{11} , A_{12} , \ldots , A_{1 3n} )$ the components of the vector
$A_1'$. Let us differential $\phi$ with respect to $A_{1l}$, for some $l=1, \ldots , 3n$.
Differentiating with respect to $A_{1l}$ variables we get form
\equ{e2} that
\begin{eqnarray*}
& & D_{A_{1l}} [ I(A_1 , A_2 ) - J_\ve ( u_0+ \tilde \phi  ) ] =
(1-\zeta )^2
 \int_0^1 t dt ( \int_{\Omega_\ve} D_{A_{1l}}
  [ ( N ( \phi +\psi )
) \phi ] \\
&+&
(p+\ve ) \int_{\Omega_\ve} \nabla_{A_{1l}} [ ( ( v_0+ \psi  +t\phi
)^{p+\ve -1} - v_0^{p+\ve -1} )
    \phi^2 ] ) .
\label{e3} \end{eqnarray*}
Using the computations in the proof of Proposition \ref{nonlin} we get that
$$
D_{A_{1l}} [ I(A_1 , A_2  ) - J_\ve ( u_0 + \tilde \psi  ) ] =
o(\ve).
$$
Now,
$$
 J_\ve ( u_0 +\hat \psi  ) -J_\ve(u_0)   =
(1-\zeta )^2 [ F_\ve ( v_0 +\psi  ) - F_\ve ( v_0 ) ]
=
$$
\be (1-\zeta)^2 \bigl\{
 \int_0^1 (1-t)dt [
 (p+\ve ) \int_{\Omega_\ve} (( v_0+ t\psi  )^{p+\ve -1} - v_0^{p+\ve -1})
\psi ^2 ]
 - 2\int_{\Omega_\ve} E \psi  \bigl\}
\label{nn}\ee where we have used that
  $$
  D F_\ve ( v_0 )[\psi ] = - \int_{\Omega_\ve} E \psi .
 $$
 Arguing as before and taking into account that
 \equ{ee41} holds, we get  \equ{ee5}.
On the other hand, using \equ{nn}, we see that
\begin{eqnarray*}
D_{A_{1l}} [  J_\ve ( u_0 +\hat \psi  ) -J_\ve(u_0)  ] & = &
(1-\zeta )^2 D_{A_{1l}} \bigl\{
 \int_0^1 (1-t)dt \times \\
& &  [
 (p+\ve ) \int_{\Omega_\ve} (( v_0+ t\psi  )^{p+\ve -1} - v_0^{p+\ve -1})
\psi ^2 ]  -
  2\int_{\Omega_\ve} E \psi  \bigl\} \\
& =&
o(\ve) -2 D_{A_{1l}} ( \int_{\Omega_\ve} E
\psi  ) .
\end{eqnarray*}
On the other hand, we have that
$$
D_{A_{1l} } ( \int_{\Omega_\ve} R^\ve \psi  )
= \left\{ \begin{array}{ll}
                         O(\ve^{2-{1 \over n-2} }) &\mbox{ if } n\leq 5 \\
                        O(\ve^{7\over 4} |\log \ve |  ) &\mbox{ if } n=6\\
                        O(\ve^{1+{4\over n-2} - {1\over n-2}}) &\mbox{ if } n\geq 7.
                          \end{array}
                \right.
$$
This concludes the proof of the Proposition.

\end{proof}

\section{The min-max} \label{minmaxsec}

In this section we set up a min-max scheme to find  a critical
point of the function $\Psi$, defined in \equ{psi1}.

We write
$$
\Psi (\Lambda , \xi , a , \theta ) =
\Psi (\Lambda_1 , \Lambda_2 , \xi_1 , \xi_2, a_1 , a_2 , \theta_1 , \theta_2 ),
$$
where
$$
(\Lambda , \xi , a , \theta ) \in \R_+^2 \times (\Omega \times \Omega \setminus \{ \xi_1 = \xi_2 \} )
\times B^2  \times K^2
$$
where
$$
B = \{ (x_1 , x_2 ) \, :\, \sqrt{x_1^2 + x_2^2 } \leq {1\over 2} \} , \quad
$$
and $K$ is a compact manifold of dimension $2n-3$, without boundary.

\medskip
\noindent
Recall that $\Omega = {\mathcal D} \setminus \omega$, where
$\omega \subset \bar B (0, \delta ) \subset {\mathcal D}$.
Define
$$
\varphi (\xi_1 , \xi_2 ) =
H(\xi_1 , \xi_1 )^{1\over 2} \, H(\xi_2  , \xi_2 )^{1\over 2} - G(\xi_1 , \xi_2 ) .
$$
 The following result holds true (see Corollary
2.1, \cite{dfm})

\begin{corollary}
\label{buco} For any (fixed) sufficiently small $\sigma>0$ there
exists $\delta_0>0$ such that for any $\delta\in(0,\delta_0)$  and
for any smooth domain $\omega\subset B(0,\delta)$ it holds
$$\varphi(\xi_1 , \xi_2 ) <0\qquad\forall\  (\xi_1 , \xi_2 ) \in  S,$$
where the manifold $S$ is defined by
$$S=\{(\xi_1,\xi_2)\in\Omega^2\ |\ |x_1|=|x_2|=\sigma\}.$$
\end{corollary}

For any $\xi = (\xi_1 , \xi_2 ) \in  S $
we let $d(\xi)=(d_1(\xi),d_2(\xi))\in \R_+^2$ be the negative
direction of the quadratic form defining $\Psi$.  We easily see that there is a
constant $c>0$ so that $c<d_1(\xi)d_2(\xi)<c^{-1}$ for all $\xi \in S $.

In the following we will construct a critical point of ``min-max" type of the function $\Psi.$
This construction has similarities with the ones developed in \cite{dfm} and \cite{mp}.
We strat with the observation that the functions
$$
a \to \Psi , \quad \theta \to \Psi
$$
have a maximum respectively in $B^2$ and $K^2$.
Let us now introduce  for $l>0$ and $\rho>0$ the following
manifold
$$ W^l_{ \rho} =\{\xi \in\Omega^2\ |\ \varphi (\xi ) < -l\}\cap V_\rho,$$
where
$$V_\rho=\{(\xi_1,\xi_2)\in\Omega ^2 \ |\ {\rm
dist}(\xi_1,\partial\Omega)>\rho,\ {\rm
dist}(\xi_2,\partial\Omega)>\rho,\ |\xi_1-\xi_2|> \rho\}.
$$
If we take
$\lambda_0=-\max\limits_{\xi \in S} \varphi (\xi )$ and $
\rho_0={\rm dist}(S,\partial \Omega),$ then
 for
any $\rho\in(0,\rho_0)$ and $l\in(0,l_0)$  we have that
$S\subset W^l_\rho.$
Moreover, for any  $R> 1$
\begin{eqnarray}
\label{disu1}
 \max\limits_{\xi \in S , \,  R^{-1} \le r\le R \atop a \in B^2 , \, \theta \in K^2}
 \Psi(r d(\xi ), \xi , a , \theta )> \max\limits_ {x\in S, \,  r=
R^{-1} ,R , \atop a \in \partial (B^2)  , \, \theta \in K^2}\Psi(r d(\xi ), \xi , a , \theta ),
\end{eqnarray}
where $d(\xi )=\big(d_1(\xi ),d_2(\xi )\big)\in\R^2_+$ is  the negative
direction of the quadratic form defining $\Psi$. This is a direct consequence of
 Corollary \ref{buco}.

Now let $A$ and $B$ be fixed numbers defined as follows
\begin{eqnarray}
\label{disu0}
 B &=& \max\limits_{\xi \in S, \,  R^{-1} \le r\le R \atop a \in B^2 , \, \theta \in K^2}
 \Psi(r d(\xi ), \xi , a , \theta ) >A  \nonumber \\
&>&  \max\limits_ {\xi \in S , \,  r=
R^{-1} ,R , \atop a \in  B^2 , \, \theta \in K^2}\Psi(r d(\xi ), \xi , a , \theta ) >
 \max\limits_ {\xi \in S , \,  r=
R^{-1} ,R , \atop a \in \partial B^2 , \, \theta \in K^2}\Psi(r d(\xi ), \xi , a , \theta ).
\end{eqnarray}
There exists $R>0$ large such that for any $l\in(0,l_0)$
it holds
\begin{eqnarray}
\label{disu}
& &B= \max\limits_{\xi \in S, \,  R^{-1} \le r\le R \atop a \in B^2 , \, \theta \in K^2}
 \Psi(r d(\xi ), \xi , a , \theta ) \ge
\max\limits_{\xi \in S, \,  \Lambda \in I  \atop a \in B^2 , \, \theta \in K^2}
 \Psi(\Lambda , \xi , a , \theta )
\nonumber\\
& &
\ge
\max\limits_{\xi \in W^l_\rho, \, \Lambda \in I  \atop a \in B^2 , \, \theta \in K^2}
 \Psi(r d(\xi ), \xi , a , \theta )
> A> \max\limits_ {\xi \in S, \,  r=
R^{-1} ,R , \atop a \in \partial B^2 , \, \theta \in K^2}\Psi(r d(\xi ), \xi , a , \theta ),
\end{eqnarray}
where $I$ is the  hyperbola in $\R^2_+$  defined by
$I=\{\Lambda\in\R^2_+\ |\ \Lambda_1\Lambda_2=1 \}.$
Indeed, for any $\Lambda\in I,$ we have
\begin{eqnarray}
\label{disu2} \Psi(\Lambda , \xi , a , \theta )
\ge-G(\xi_1 ,\xi_2)
\ge -{1\over \rho^{n-2}}\tau >A,
\end{eqnarray}
provided that $R$ is choosen properly.

\begin{lemma}
\label{num.2}

There exist $l_0>0$ and $\rho_0>0$   such that for any
$l\in(0,l_0)$ and $\rho\in(0,\rho_0)$ the function $\Psi$
 satisfies the following property:

 for any sequence $(\Lambda_n , \xi_n , a_n , \theta_n )$  in $  [R^{-1} , R]^2 \times W^l_\rho \times B_2^2 \times K^2$
such that $$\lim\limits_n(\Lambda_n , \xi_n , a_n , \theta_n )=(\Lambda , \xi , a, \theta
)\in\partial( [R^{-1} , R]^2 \times W^l_\rho \times B^2 \times K^2  )$$  and $ \psi(\Lambda_n ,
\xi_n , a_n , \theta_n ) \in[A,B]$ there exists a vector $T$ tangent to $\partial( [R^{-1} , R]^2 \times W^l_\rho \times B^2 \times K^2  )$ at $(\xi , \Lambda , a , \theta),$ such that
$$\nabla \Psi(\xi , \Lambda , a , \theta )\cdot T\ne0.$$
\end{lemma}
\begin{proof}
First of all we observe  that $\Lambda_n$ is component-wise bounded
from below and from above by a positive constant. In fact, if
$|\Lambda_n|\to+\infty$ and $|\Lambda_n|\to 0$ then $|\psi(\Lambda_n , x_n)|\to+\infty$, which is impossible. Thus we have that $\Lambda \not\in \{ R^{-1} , R \}$. Furthermore, since $a \to \Psi$ has a maximum in $B$, if $a \in \partial B^2$, then $\nabla_a \Psi \not=0$, and we can choose $T=\nabla_a \Psi$.
Similarly, if $\nabla_\Lambda\psi(\Lambda , x)\ne 0,$ then $T$ can be chosen
parallel to $\nabla_\Lambda\psi(\Lambda , x)$. Then assume that
$\nabla_\Lambda\psi(\Lambda , x)= 0,$
$ \Lambda$ satisfies
$$
\Lambda_1^2 = -{H(\hat \xi_2 ,\hat \xi_2)^{1/2}  \over
H(\hat \xi_1 , \hat \xi_1)^{1/2} \varphi (\hat \xi_1,\hat \xi_2)},
\quad  \Lambda_2^2 = -{H(\hat \xi_1 , \hat \xi_1)^{1/2}  \over
H(\hat \xi_2 , \hat \xi_2)^{1/2} \varphi (\hat \xi_1,\hat \xi_2)},
$$
and $\hat \xi$ satisfies $\varphi(\hat \xi)<0$. Substituting back
in $\Psi$, we get
$$ \Psi= - {1\over 2} + {1\over 2} \log {1\over
|\varphi (\hat \xi_1 , \hat \xi_2)|}.
$$
Thus we conclude the proof of the Lemma, using the following result proved in \cite{dfm}.

\begin{lemma}
\label{phi} Given $c<0 $ there exists a sufficiently small number
$\rho
>0$ with the following property: If $(\bar \xi_1, \bar \xi_2) \in
\partial( \Omega_\rho\times  \Omega_\rho)$ is such that $\varphi (\bar \xi_1, \bar
\xi_2) =  c$, then there is a vector ${\bf \tau}$, tangent to
$\partial ( \Omega_\rho\times  \Omega_\rho)$   at the point $(\bar
\xi_1, \bar \xi_2)$, so that \be \nabla \varphi (\bar \xi_1, \bar
\xi_2)\cdot {\bf \tau} \ne 0 .\label{rhorho} \ee The number $\rho$
does not depend on $c$.
\end{lemma}
\end{proof}

We now have the tools to show the validity of the following fact

\begin{proposition}
\label{liv-cri} There exists   a critical level for $\Psi $
between $A$ and $B.$
\end{proposition}
\begin{proof} First we claim that the function $\Psi$ constrained to $\R^2_+ \times
W^l_\rho \times B_2^2 \times K^2 $ satisfies the Palais-Smale condition in $[A,B]$.
Indeed, let  $(\Lambda_n , \xi_n , a_n , \theta_n )$  in $ \R^2_+ \times W^l_\rho \times B^2 \times K^2 $ be
such that $\lim\limits_n\Psi(\Lambda_n , \xi_n , a_n , \theta_n  )\in [A , B]$ and
$\lim\limits_n\nabla\Psi(\Lambda_n , \xi_n , a_n , \theta_n )=0.$ Arguing as in the
proof of Lemma \ref{num.2} it can be shown that $\Lambda_n$
remains bounded component-wise from above and below by a positive
constant.

Assume now by contradiction that there are no critical
levels in the interval $[A,B].$ We can define an appropriate
negative gradient flow that will remain in $ [R^{-1} , R]^2 \times
W^l_\rho \times B_2^2 \times K^2 $  at any level $c\in[A,B].$ Moreover the Palais-Smale
condition holds in $[A,B]$. Hence there exists a continuous
deformation
$$\eta:[0,1]\times\Psi^B\to\Psi^B$$ such that for some $A'\in(0,A)$
\begin{eqnarray*}
& &\eta(0,u)=u\qquad\forall\ u\in\Psi^B\\
& &\eta(t,u)=u\qquad\forall\ u\in\Psi^{A'}\\
& &\eta(1,u) \in\Psi^{A'}.\\
\end{eqnarray*}
Let us call
$${\mathcal A}=\{(\Lambda, \xi , a , \theta )\in \R^2_+\times W^l_\rho \times B^2 \times K^2  \ |\ \xi \in S,\ \Lambda=rd(\xi ),\ R^{-1} \le r\le R\},$$
$$\partial {\mathcal A}=\{(\Lambda, \xi , a , \theta )\in \R^2_+\times W^l_\rho  \times B^2 \times K^2  \ |\ \xi \in S,\ \Lambda=R^{-1} \ or\ \Lambda=Rd(\xi )\},$$
$${\mathcal C}= I_\tau \times W^l_\rho.$$
>From (\ref{disu})  we deduce that  ${\mathcal A} \subset\Psi^B,$
$\partial {\mathcal A} \subset\Psi^{A'}$ and $\Psi^{A'}\cap{\mathcal
C}=\emptyset.$ Therefore
\begin{eqnarray}
& &\eta(0,u)=u\qquad\forall\ u\in {\mathcal A},\nonumber\\
& &\eta(t,u)=u\qquad\forall\ u\in\partial {\mathcal A},\nonumber\\
& &\eta(1,{\mathcal A})\cap {\mathcal C}=\emptyset.
\end{eqnarray}
For any $(\Lambda, \xi , a , \theta )\in  {\mathcal A}$ and for any $t\in[0,1]$ we
denote
$$\eta\big(t,(\Lambda, \xi , a ,
\theta )\big)=
\big(\widetilde \Lambda , \widetilde \xi , \widetilde a , \widetilde \theta
\big)\in \R^2_+ \times W^l_\rho \times B^2 \times K^2.$$
 We define the set
$$ {\mathcal B}  =\{(\Lambda, x)\in{\mathcal A}\ |\ \widetilde \Lambda  \in I \}.$$
Since $\eta(1,{\mathcal A})\cap {\mathcal C}=\emptyset $ it holds $ {\mathcal
B}  =\emptyset .$ Now let  ${\mathcal U}$  be a neighborhood of $
{\mathcal B} $ in $\R_+^2 \times W_\rho^l \times B^2 \times K^2$ such that
 $H^*({\mathcal U})= H^*( {\mathcal B} ).$
If
 $\pi:{\mathcal U}\to {\mathcal S}$ denotes the projection, arguing like in
  Lemma 7.1 of \cite{dfm} we can show that
$$ \pi^*:H^*({\mathcal S})\to H^*({\mathcal U})\quad\hbox{
is a monomorphism.}
$$
This  condition provides a contradiction, since $H^*({\mathcal U})=
\{0\} $ and  $H^*({\mathcal S})\ne\{0\}.$ \end{proof}

\newpage

\section{Appendix}\label{A}

\noindent
To give a first description of these solutions, let us introduce some notations.
Fix an integer $k$. For any integer $l=1, \ldots , k$, we define angles $\theta_l$ and vectors $\textsf{n}_l$, $\textsf{t}_l$ by
\begin{equation}
\label{defanglespoints}
\theta_l = {2\pi \over k} \, (l-1), \quad \textsf{n}_l = (\cos \theta_l , \sin \theta_l, \textsf{0} ), \quad \textsf{t}_l = (-\sin \theta_l , \cos \theta_l , \textsf{0}).
\end{equation}
Here $\textsf{0}$ stands for the zero vector in $\R^{n-2}$. Notice that  $\theta_1=0$, $\textsf{n}_1 = (1, 0, \textsf{0})$, and  $\textsf{t}_1 = (0, 1,\textsf{0})$.

\medskip
\noindent
In  \cite{dmpp1}  it was proved that there exists $k_0$ such that for all integer $k>k_0$
there exists a solution $Q= Q_k$ to \equ{eqqq} that can be described as follows
\begin{equation}
\label{sol}
Q_k(x) =U_* (x) +\tilde  \phi (x).
\end{equation}
where
\begin{equation}
\label{defU*}
U_* (x) =  U(x) - \sum_{j=1}^k U_j (x),
\end{equation}
while $\tilde \phi$ is smaller than $U_*$.
The functions $U$ and  $U_j$ are positive solutions to \equ{eqqq}, respectively defined as
\begin{equation}
\label{basic}
U(x) = \gamma \left( {2 \over 1+ |x|^2} \right)^{n-2 \over 2}, \quad U_j (x) = \mu_k^{-{n-2 \over 2}} U(\mu_k^{-1} (x-\xi_j )),
\end{equation}
where $\gamma = \left[ { n (n-2) \over 4} \right]^{n-2 \over 4}.$
For any integer $k$ large, the parameters $\mu_k  >0$ and the $k$ points $\xi_l$, $l=1, \ldots , k$
are given by
\begin{equation} \label{parameters}
\left[  \sum_{l > 1}^k {1 \over (1-\cos \theta_l )^{n-2 \over 2}} \right] \, \mu_k^{n-2 \over 2} =  \left(1+ O ({1\over k}) \right) , \quad {\mbox {for}} \quad k \to \infty
\end{equation}
in particular $ \mu_k \sim k^{-2}$ if $n\geq 4$, and $\mu_k \sim k^{-2} |\log k|^{-2}$ if $n=3$, as $k \to \infty$, and
\begin{equation} \label{parameters1}
 \xi_l = \sqrt{1-\mu^2} \, ( \textsf{n}_l , 0).
\end{equation}

The function $ \tilde \phi $ in \equ{sol}  can be further decomposed. Let us introduce  some cut-off functions $\zeta_j$
to be defined as follows.
Let $\zeta(s)$ be a smooth function such that $\zeta(s) = 1$ for $s<1$ and $\zeta(s)=0$ for $s>2$. We also let $\zeta^-(s) = \zeta(2s)$.
Then we set
$$
 \zeta_j(y) = \left\{ \begin{matrix}   \zeta(\, k\eta^{-1} |y|^{-2} |( y  -\xi|y|)\,  |\, )  & \hbox{ if } |y|> 1\, , \\ & \\     \zeta{ (\,k \eta^{-1}\,|y-\xi|\, )}   & \hbox{ if } |y|\le  1\, , \end{matrix}
  \right. \
$$
in such a way that that
$$
\zeta_j( y) = \zeta_j( y/|y|^2).
$$
The function $\tilde \phi$  has the form
\begin{equation} \label{decphi}
\tilde  \phi\  =\  \sum_{j=1}^k  \tilde \phi_j  + \psi.
\end{equation}
In the decomposition \equ{decphi} the functions $\tilde \phi_j$, for $j>1$, are defined in terms of $\tilde \phi_1$
\begin{equation}
\tilde  \phi_j (\bar y, y')= \tilde \phi_1( e^{\frac{2\pi j} k i} \bar y, y'), \quad j=1,\ldots, k-1.
\label{sim11}\end{equation}
We have that
\begin{equation}
\label{estpsi}
 \| \psi \|_{n-2} \leq C k^{1-{n\over q}} \quad {\mbox{if}} \quad n\geq 4, \quad  \| \psi \|_{n-2} \leq {C \over \log k} \quad {\mbox{if}} \quad n=3,
\end{equation}
where
\begin{equation}
 \|\phi\|_{n-2} :=
\|\,(1+ |y|^{n-2}) \phi \, \|_{L^\infty(\R^n)} \ .
\label{nomstar1}\end{equation}
On the other hand, if we rescale and translate the function $\tilde \phi_1$
\begin{equation}
\label{phiriscalata}
\phi_1 (y) =\mu^{n-2 \over 2}  \tilde \phi_1 (\xi_1 + \mu y )
\end{equation}
we have the validity of the following estimate for $\phi_1$
\begin{equation}
\label{estphi1}
\|  \phi_1 \|_{n-2} \leq C k^{-{n\over q}}  \quad {\mbox {if}} \quad n\geq 4, \quad \|  \phi_1 \|_{n-2} \leq {C \over k \log k}  \quad {\mbox {if}} \quad n= 3.
\end{equation}

In terms of the function $ \tilde \phi$ in the decomposition \equ{sol}, equation \equ{eqqq} gets re-written as
\begin{equation}
\Delta \tilde  \phi +  p\gamma|U_*|^{p-1}\tilde \phi +   E + \gamma  N(\tilde \phi) = 0
\label{eq1}\end{equation}
where $E$ is defined
by
\begin{equation}
 \label{error}
E= \Delta U_* + f(U_* )
\end{equation}
and
$$
N(\phi) =  |U_* + \phi |^{p-1} (U_* + \phi)  - |U_*  |^{p-1} U_* - p|U_*|^{p-1} \phi.
$$
One has a precise control of the size of the function $E$ when measured for instance in the following norm.
Let us fix a number $q$, with $ \frac n2 < q < n$, and consider the weighted $L^q$ norm
\begin{equation}
\|h\|_{**} =  \|  \, (1+ |y|)^{{n+2} - \frac {2n} q} h \|_{L^q(\R^n)}.
\label{nomstar2} \end{equation}

In  \cite{dmpp1} it is proved that there exist an integer $k_0$ and a positive constant $C$ such that for all $k\geq k_0$ the following estimates
hold true
\begin{equation}
\label{rivoli4}
\| E \|_{**} \leq C k^{1-{n\over q}} \quad {\mbox {if}} \quad n\geq 4, \quad
\| E \|_{**} \leq {C \over \log k}  \quad {\mbox {if}} \quad n=3
\end{equation}

\medskip

To be more precise, we have  estimates for the $\| \cdot \|_{**}$-norm of the error term $E$
 in the {\it exterior region} $\bigcap_{j=1}^k \{ |y-\xi_j| > {\eta \over k}  \} $, and also in the {\it interior regions} $ \{ |y-\xi_j| < {\eta \over k}  \} $, for any $j=1, \ldots , k$. Here $\eta >0$ is a positive and small constant,  independent of $k$.

\medskip

\medskip
\noindent {\it In the exterior region}. We have
$$
\|  \, (1+ |y|)^{{n+2} - \frac {2n} q} E (y)  \|_{L^q(\bigcap_{j=1}^k \{ |y-\xi_j| > {\eta \over k}  \})} \leq C k^{1-{n\over q}}
$$
if $n\geq 4$, while
$$
\|  \, (1+ |y|)^{{n+2} - \frac {2n} q} E (y)  \|_{L^q(\bigcap_{j=1}^k \{ |y-\xi_j| > {\eta \over k}  \})} \leq {C \over \log k}
$$
if $n=3$.

\medskip
\noindent {\it In the interior regions}.
Now, let $|y-\xi_j| < {\eta \over k} $ for some $j \in \{ 1, \ldots , k\}$ fixed.
It is convenient to measure the error after a change of scale.
Define
$$
\tilde E_j(y) :=  \mu^{\frac {n+2}2} E ( \xi_j + \mu y ) , \quad |y| < \frac \eta {\mu  k}
$$
We have
$$
\|  \, (1+ |y|)^{{n+2} - \frac {2n} q}\tilde E_j (y) \|_{L^q( |y-\xi_j| <{\eta \over \mu k })  } \leq C k^{-{n\over q}} \quad {\mbox {if}} \quad n\geq 4
$$
and
$$
\|  \, (1+ |y|)^{{n+2} - \frac {2n} q}\tilde E_j (y) \|_{L^q( |y-\xi_j| <{\eta \over \mu k })  } \leq {C \over k \log k} \quad {\mbox {if}} \quad n=3.
$$
We refer the readers to \cite{dmpp1}.

\medskip
\noindent

Let us now define the following functions
\begin{equation}
\label{ang2}
 \begin{array}{rr}
\pi_\alpha (y)  = {\partial \over \partial y_\alpha } \tilde  \phi (y)  ,  & \quad   {\mbox {for}} \quad \alpha = 1, \ldots , n; \\
 & \\
 \pi_0 (y)   ={n-2 \over 2}\tilde  \phi (y) + \nabla \tilde \phi (y) \cdot y.  &
\end{array}
\end{equation}
In the above formula $\tilde \phi$ is the function defined in \equ{sol} and described in \equ{decphi}.
Observe that the function $\pi_0 $ is even in each of its variables, namely
$$
\pi_0 ( y_1, \ldots , y_j , \ldots , y_n) = \pi_0 ( y_1, \ldots ,- y_j , \ldots , y_n)\quad \forall j=1, \ldots , n,
$$
while $\pi_\alpha$, for $\alpha = 1, \ldots , n$ is odd in the $y_\alpha$ variable, while it is even in all the other variables. Furthermore, all functions $\pi_\alpha$ are invariant under rotation of ${2\pi \over k}$ in the first two coordinates, namely they satisfy \equ{sim00}.
The functions $\pi_\alpha$ can be further described, as follows.

The functions $\pi_\alpha$ can be decomposed into
\begin{equation}
\label{martes4}
\pi_\alpha (y) = \sum_{j=1}^k \tilde \pi_{\alpha , j} (y) + \hat \pi_\alpha (y)
\end{equation}
where
$$
\tilde \pi_{\alpha ,  j} (y) = \tilde \pi_{\alpha , 1} (e^{ {2\pi \over k} \, j \, i} \bar y , y').
$$
Furthermore, there exists a positive constant $C$ so that
$$
\| \hat \pi_0 \|_{n-2} \leq C k^{1-{n\over q}}, \quad \| \hat \pi_j \|_{n-1} \leq C k^{1-{n\over q}}, \quad j=1, \ldots , k,
$$
if $n\geq 4$, and
$$
\| \hat \pi_0 \|_{n-2} \leq { C \over \log k} , \quad \| \hat \pi_j \|_{n-1} \leq  { C \over \log k}, \quad j=1, \ldots , k,
$$
if $n=3$. Furthermore,  if we denote $\pi_{\alpha , 1} (y) = \mu^{n-2 \over 2} \tilde \pi_{\alpha , 1} (\xi_1 + \mu y )$, then
$$
\|  \pi_{0 , 1}  \|_{n-2} \leq C k^{-{n\over q}}, \quad \|  \pi_{\alpha , 1}  \|_{n-1} \leq C k^{-{n\over q}}, \quad \alpha=1, \ldots , n
$$
if $n\geq 4$, and
$$
\|  \pi_{0 , 1}  \|_{n-2} \leq { C \over k \log k} , \quad \|  \pi_{\alpha , 1}  \|_{n-1} \leq C { C \over k \log k}, \quad \alpha=1, \ldots , 3
$$
if $n = 3$.

For further details we refer the interested reader to
 \cite{dmpp1}.


\begin{thebibliography}{AAA}


            \bibitem{aubin}
T. Aubin,  Probl\`emes isop\'erimetriques et espaces de Sobolev, {\em J. Differ. Geometry} 11 (1976), 573--598



\bibitem {BC} A. Bahri, J.M. Coron, On a nonlinear elliptic equation
involving the critical Sobolev exponent: the effect of the topology of the
domain. \emph{Comm. Pure Appl. Math.} \textbf{41} (1988), 253--294.

\bibitem{benbou} M. Ben Ayed, K.O. Bouh,  Nonexistence results of sign-changing solutions to a supercrtiical nonlinear
problem. {\em Commun. Pure Appl. Anal. } 7(5), 1057–-1075 (2008).

\bibitem{bengr} M. Ben Ayed, K. El Mehdi, M. Grossi, O. Rey, A nonexistence result of single peaked solutions to a
supercritical nonlinear problem. {\em Commun. Contemp. Math.} 5(2), 179–-195 (2003).

\bibitem{CGS} L.A. Caffarelli, B. Gidas, J. Spruck, Asymptotic symmetry and local behavior
of semilinear elliptic equations with critical Sobolev growth, {\em Comm. Pure
Appl. Math.} 42 (1989), 271-297.


\bibitem {cw1} M. Clapp, T. Weth, Minimal nodal solutions of the pure
critical exponent problem on a symmetric domain, \emph{Calc. Var. }\textbf{21}
(2004), 1--14.

\bibitem {cw} M. Clapp, T. Weth, Two solutions of the Bahri-Coron
problems in punctured domains via the fixed point transfer, \emph{Commun.
Contem. Math.}

\bibitem{cmp} M. Cl\'app, M. Musso, A. Pistoia.
Multipeak solutions to the pure critical exponent problem in
punctured domains. {\em Journal of Functional Analysis}, 256 (2009), no.
2, 275--306.

\bibitem {Co} J.M. Coron, Topologie et cas limite des injections de
Sobolev. \emph{C. R. Acad. Sci. Paris Ser. I Math.} \textbf{299} (1984), 209--212.


 \bibitem{Da}
 E.N. Dancer,
 {\em A note on an equation with critical exponent     }, Bull. London Math. Soc.
{\bf 20} (1988), 600-602.

\bibitem{D} W. Ding, On a conformally invariant elliptic equation on $R^n$, {\em Communications on Mathematical Physics }  107(1986), 331-335.




 \bibitem{Di}
 W. Ding,
 Positive solutions of $\Delta u + u^{N+2 \over N-2} =0$
 on contractible domains,
{\em  J. Partial Differential Equations } {\bf 2}, no. 4 (1989), 83-88.



\bibitem{dfm} M. del Pino,  P. Felmer,  M. Musso.  Two-bubble
solutions in the super-critical Bahri-Coron's problem.  {\em Calculus of
Variations and PDE} 16, no. 2 (2003), 113--145.



\bibitem{dfm1}
M. del Pino, P. Felmer, M. Musso,  Multi-peak solutions for
super-critical elliptic problems in domains with small holes. {\em J.
Differential Equations } 182 (2002), no. 2, 511--540.

\bibitem{dfm2}
M. del Pino, P. Felmer, M. Musso
 Multi-Bubble Solutions for Slightly Super-Critial Ellitpic Problems
in Domains with Symmetries. {\em Bullettin of the London Mathematical Society} 35, no. 4
(2003), 513--521.

\bibitem{dmpp1}  M. del Pino, M. Musso, F. Pacard, A. Pistoia.  Large Energy Entire Solutions for the Yamabe Equation. {\em Journal of Differential Equations} 251,  (2011), no. 9,  2568--2597.

\bibitem{dmpp2}  M. del Pino, M. Musso, F. Pacard. A. Pistoia.  Torus action on $S^n$ and sign changing solutions for conformally invariant equations. {\em Annali della Scuola Normale Superiore di Pisa,  Cl. Sci.} (5) 12 (2013), no. 1, 209--237.

\bibitem{DKM} T. Duyckaerts, C. Kenig and F. Merle, Solutions of the focusing nonradial critical wave equation with the compactness property,  arxiv:1402.0365v1



\bibitem {KW} J. Kazdan, F.W. Warner, Remarks on some quasilinear
elliptic equations. \emph{Comm. Pure Appl. Math.} \textbf{28} (1975), 567--597.


\bibitem{mp}
M. Musso, A. Pistoia.  Multispike solutions for a nonlinear
elliptic problem involving critical Sobolev exponent. {\em Indiana
University Mathematical Journal} 51, no. 3 (2002),541--579.


\bibitem{mw} M. Musso, J. Wei,  Nondegeneracy of Nonradial  Nodal Solutions to Yamabe
 Problem. Preprint 2014.

\bibitem{obata}
M. Obata, {\em Conformal changes of Riemannian metrics on a Euclidean sphere}. Differential geometry (in honor of Kentaro Yano), pp. 347--353. Kinokuniya, Tokyo, (1972).

\bibitem{Pa1}
D. Passaseo,
 Multiplicity of positive solutions of nonlinear elliptic equations with critical
 Sobolev exponent in some contractible domains.{\em  Manuscripta Math. } 65 (1989), no. 2, 147--165.

\bibitem{Pa2}
D. Passaseo,  New nonexistence results for elliptic equations
with supercritical nonlinearity, {\em Differential and Integral
Equations}, {\bf 8}, no. 3 (1995), 577-586.

\bibitem{pr} A. Pisotia, O. Rey, Multiplicity of solutions to the supercritical Bahri-Coron's problem in pierced
domains. {\em Adv. Diff. Equ.} 11(6), 647–666 (2006).

\bibitem{PV}  Pistoia, A., Vétois, J. Sign-changing bubble towers for asymptotically critical
elliptic equations on Riemannian manifolds. {\em J. Diff. Eqs.} to appear.


\bibitem{Po}
S. Pohozaev,  Eigenfunctions of the equation $\Delta u + \lambda f(u) = 0$, {\em Soviet. Math. Dokl.} 6, (1965), 1408--1411.


\bibitem{RV1}  Robert, Frédéric; Vétois, Jérôme, Examples of non-isolated blow-up for perturbations of the scalar curvature equation on non-locally conformally flat manifolds. {\em J. Differential Geom.} 98 (2014), no. 2, 349–356.

    \bibitem{RV2}  Robert, Frédéric; Vétois, Jérôme Sign-changing blow-up for scalar curvature type equations. {\em Comm. Partial Differential Equations} 38 (2013), no. 8, 1437–1465


\bibitem{talenti}
G. Talenti,  Best constants in Sobolev inequality, {\em Annali di Matematica } 10 (1976), 353--372.

\bibitem{vaira}
G. Vaira, {A new kind of blowing-up solutions for the
Brezis-Nirenberg problem}, to appear in Calculus of Variations and PDEs.










\end{thebibliography}
\end{document}